\documentclass[a4paper,11pt]{amsart}
\usepackage{amsfonts,amsthm,amsmath,amssymb,graphicx}
\theoremstyle{plain}
\newtheorem{lem}{Lemma}[section]
\newtheorem{prop}[lem]{Proposition}
\newtheorem{thm}[lem]{Theorem}

\theoremstyle{definition}

\theoremstyle{remark}

\DeclareMathOperator{\rank}{rank}

\DeclareMathOperator{\sym}{sym}
\DeclareMathOperator{\spn}{span}

\DeclareMathOperator{\cond}{cond}

\newcommand{\Z}{\mathbb Z}

\newcommand{\A}{\mathbb A}
\newcommand{\F}{\mathbb F}
\newcommand{\E}{\mathbb E}
\newcommand{\R}{\mathbb R}
\newcommand{\C}{\mathbb C}
\newcommand{\stufe}{\mathcal N }
\newcommand{\G}{\mathcal G}
\newcommand{\Eis}{\mathcal E}
\newcommand{\h}{\mathcal H}
\newcommand{\M}{\mathcal M}

\begin{document}

% Enter full title and short title for running headers
\title[Degree 2 Eisenstein series]{Hecke eigenvalues and relations for degree 2 Siegel Eisenstein series}
%\shorttitle{Degree 2 Eisenstein series}

% Enter the publication year and the ID number of the paper
%\volumeyear{}
%\paperID{}

% Author name(s)
\author{Lynne H. Walling}
\address{School of Mathematics, University of Bristol, University Walk, Clifton, Bristol BS8 1TW, United Kingdom}
\email{l.walling@bristol.ac.uk}
% Abbreviated author name for running headers
%\abbrevauthor{L.H. Walling}
% Abbreviated author name for first page header
%\headabbrevauthor{Walling, L.H.}

%\address{School of Mathematics, University of Bristol, University Walk, Clifton, Bristol BS8 1TW, United Kingdom}

% Address / e-mail address of corresponding author
%\correspdetails{l.walling@bristol.ac.uk}

% Enter received/revised/accepted dates as necessary
%\received{}
%\revised{}
%\accepted{}

% Enter details of editor communicating this article
%\communicated{A. Editor}

\begin{abstract} We evaluate the action of Hecke operators on Siegel Eisenstein series of degree 2,
square-free level $\stufe$ and arbitrary character
$\chi$, without using knowledge of their Fourier coefficients.  From this
we construct a basis of simultaneous eigenforms for the full Hecke algebra, and we 
compute their eigenvalues.  As well, we obtain
Hecke relations among the Eisenstein series.  Using these Hecke relations
in the case that $\chi=1$, we generate
a basis for the space of Eisenstein series from $\E_{(\stufe,1,1)}.$
\end{abstract}

\maketitle
\def\thefootnote{}
\def\thefootnote{\arabic{footnote}}

\section{Introduction} 
Modular forms are of central interest in number theory, particularly
because Hecke theory tells us that their Fourier coefficients carry
number theoretic information.
Eisenstein series are fundamental examples of modular forms
and play an important
role in the theory of modular forms.  
In the case of elliptic modular forms (i.e. Siegel degree 1), the
Eisenstein series are well-understood; for instance, we have explicit
formulas for their Fourier coefficients, we know that the space of
Eisenstein series can be simultaneously diagonalised with respect to
the Hecke operators attached to primes not dividing the level, and when the
level is square-free, the space can be simultaneously diagonalised with
respect to the full Hecke algebra.  In the case of Siegel degree $n>1$,
the situation is less well understood, but some parallel results have
been established.
Freitag \cite{F} has shown that the space of Eisenstein series can be simultaneously
diagonalised with respect to the Hecke operators attached to primes not
dividing the level.  Many authors have worked on computing Fourier coefficients
of Siegel Eisenstein series with $n>1$.  We do not try to give a comprehensive
list of all the work that has contributed to this, but rather give a sampling.
For level 1, see \cite{Maass1},\cite{Maass2} for degree 2; \cite{Kat1} for degree 3;
\cite{B},\cite{CK},\cite{Kat2},\cite{Koh} for arbitrary degree.
For degree 2, level $\stufe$ and primitive character modulo $\stufe$,
Fourier coefficients for 1 of the Eisenstein series
$\E_{(\stufe,1,1)}$ have been computed
in \cite{Miz} when $\stufe$ is odd and square-free, and in \cite{Tak} for arbitrary $\stufe$.

In this work, without using any knowledge of Fourier coefficients,
we evaluate the
action of Hecke operators on a basis for the space of degree 2 Siegel Eisenstein series of square-free
level $\stufe$ and arbitrary character $\chi$ (Propositions 3.3-3.10). 
The evaluation of Hecke operators reveals
 Hecke relations among these Siegel Eisenstein series.  Using these,
we construct a basis for this space consisting of simultaneous
eigenforms for the full Hecke algebra, and we compute their eigenvalues (Theorem 3.11).
In the case that $\cond\chi^2<\stufe$, these Hecke relations allow us to
generate some of the other Eisenstein series from $\E_{(\stufe,1,1)}$
by applying particular elements of the Hecke algebra;
in particular, when $\chi=1$, we can generate a basis from $\E_{(\stufe,1,1)}$
(Theorem 3.12).  In the remark following this theorem, we briefly discuss
how we can use \cite{HW} and the Fourier coefficients of the degree 2, level 1 Eisenstein series $\E$
to generate the Fourier coefficients of all the 
degree 2, level $\stufe$ Eisenstein series
in the case that $\stufe$ is square-free and the character $\chi=1$.

\section{Preliminaries}
Here we set notation and define degree 2 Siegel Eisenstein series and
Hecke operators.  We begin by fixing square-free $\stufe\in\Z_+$.
With $Sp_2(\Z)$ the group of $4\times 4$ integral symplectic
matrices, we set
$$\Gamma_{\infty}=\left\{\left(\begin{array}{cc} *&*\\0&*\end{array}\right)\in
Sp_2(\Z)\right\},$$
$$\Gamma_0(\stufe)=\left\{\gamma\in Sp_2(\Z):\ 
\gamma\equiv\left(\begin{array}{cc} *&*\\0&*\end{array}\right)\ (\stufe)\ \right\}.$$
The 0-dimensional cusps for $\Gamma_0(\stufe)$ correspond to the elements of the double coset
$\Gamma_{\infty}\backslash Sp_2(\Z)/\Gamma_0(\stufe)$.
For $k\in\Z_+$ and $\chi$ a Dirichlet character modulo $\stufe$, we have one
Siegel Eisenstein series for each cusp, defined as follows.
For $\gamma_0\in Sp_2(\Z)$, the Eisenstein series associated to the 
cusp $\Gamma_{\infty}\gamma_0\Gamma_0(\stufe)$ is 
$$\E_{\gamma_0}(\tau)=\sum\overline\chi(\det D_{\gamma})\,1|\gamma(\tau)$$
where $\Gamma_{\infty}\gamma$ varies over the $\Gamma_0(\stufe)$-orbit of
$\Gamma_{\infty}\gamma_0$, 
$$\tau\in\h_{(2)}=\left\{X+iY:\ X,Y\in\R^{2,2}_{\sym},\ Y>0\ \right\}$$
where $Y>0$ denotes that $Y$ is the matrix for a positive definite quadratic form,
and
$$1|\left(\begin{array}{cc} A&B\\C&D\end{array}\right)(\tau)=\det(C\tau+D)^{-k}.$$
This sum is well-defined provided $\chi_q^2=1$ whenever $q$ is a prime dividing $\stufe$ and
$\rank_q M_0=1$ where
$\gamma_0=\left(\begin{array}{cc} *&*\\M_0&N_0\end{array}\right)$ 
and $\rank_qM_0$ denotes the rank of $M_0$ modulo $q$.
When well-defined, the sum is non-zero provided $\chi(-1)=(-1)^k$, and it is
 absolutely uniformly convergent on compact regions provided $k\ge 4$
(and hence it is analytic, meaning analytic in each variable of $\tau$).
For $\gamma'\in\Gamma_0(\stufe)$, $\Gamma_{\infty}\gamma\gamma'$ varies
over the $\Gamma_0(\stufe)$-orbit of $\Gamma_{\infty}\gamma_0$ as
$\Gamma_{\infty}\gamma$ does, and hence
$\E_{\gamma_0}|\gamma'=\chi(\det D_{\gamma'})\,\E_{\gamma_0}.$
As noted in [3], these Eisenstein series are linearly independent,
and the 0th Fourier coefficient of $\E_{\gamma_0}$ is 0 unless $\gamma_0\in\Gamma_0(\stufe)$,
in which case it is 1.

A pair of $2\times 2$ matrices $(M\ N)$ is 
called symmetric if $M\,^tN=N\,^tM$ with $^tN$ denoting the transpose of $N$;
it is called a coprime pair if $M,N$ are integral and $(GM\ GN)$ 
is integral only if $G$ is.  Note that $(M\ N)$ is a coprime pair if and only if,
for each prime $p$, $\rank_p(M\ N)=2$.
It is well-known that for $\gamma,\gamma'\in Sp_2(\Z)$, $\gamma$ and $\gamma'$ lie
in the same coset in $\Gamma_{\infty}\backslash Sp_2(\Z)$ if and only if
$\gamma=\left(\begin{array}{cc} *&*\\M&N\end{array}\right)$, $\gamma'=\left(\begin{array}{cc} *&*\\GM&GN\end{array}\right)$
for some $G\in GL_2(\Z)$.  Thus these cosets can be parameterised by
$GL_2(\Z)$-equivalence classes of coprime symmetric pairs; so
$\E_{\gamma_0}$ is supported on a set of $GL_2(\Z)$-equivalence
class representatives for the
$\Gamma_0(\stufe)$-orbit of $GL_2(\Z)(M_0\ N_0)$.

For each prime $p$, we have Hecke operators $T(p)$ and $T_1(p^2)$ that
act on degree 2 Siegel modular forms, and
$\{T(p),T_1(p^2):\ p\text{ prime }\}$ generates the Hecke algebra.
For $f$ a degree 2 Siegel modular form of weight $k$, level $\stufe$, and
character $\chi$, and for $\gamma'=\left(\begin{array}{cc} A&B\\C&D\end{array}\right)$, we set
$$\gamma'\circ\tau=(A\tau+B)(C\tau+D)^{-1}$$
and
$$f(\tau)|\gamma'=(\det\gamma')^{k/2}\det(C\tau+D)^{-k}f(\gamma'\circ\tau).$$  Then
$$f|T(p)=p^{k-3}\sum_{\gamma}\overline\chi(\det D_{\gamma})f|\delta^{-1}\gamma$$
where $\delta=\left(\begin{array}{cc} pI_2\\&I_2\end{array}\right)$ and $\gamma$ varies over a set
of coset representatives for
$$\left(\delta\Gamma_0(\stufe)\delta^{-1}\cap\Gamma_0(\stufe)\right)\backslash\Gamma_0(\stufe).$$
Somewhat similarly,
$$f|T_1(p^2)=p^{k-3}\sum_{\gamma}\overline\chi(\det D_{\gamma})f|\delta_1^{-1}\gamma$$
where $\delta_1=\left(\begin{array}{cc} X\\&X^{-1}\end{array}\right)$, $X=\left(\begin{array}{cc} p\\&1\end{array}\right)$, and $\gamma$
varies over a set of coset representatives for
$$\left(\delta_1\Gamma_0(\stufe)\delta_1^{-1}\cap\Gamma_0(\stufe)\right)
\backslash\Gamma_0(\stufe).$$
In Propositions 2.1 and 3.1 of [4] we computed an explicit set of upper triangular block matrices giving
the action of the Hecke operators, and we will use these here in evaluating
the action of Hecke operators on Eisenstein.  (Note that
in [4] we did not introduce the normalisation of $T_1(p^2)$ until we averaged
the Hecke operators to produce an alternative basis for the Hecke algebra.)

Given $Q\in\Z^{2,2}_{\sym}$ and $\F=\Z/p\Z$, $p$ prime, we can think of $Q$ as
a quadratic form on $V=\F x_1\oplus\F x_2$.  We say a non-zero vector $v\in V$
is isotropic if $Q(v)=0$ (in $\F$).  Suppose $p$ is odd.  Then $Q$ is a
$GL_2(\F)$ conjugate of $\mathbb H=\big<1,-1\big>$ or of $\A=\big<1,-\omega\big>$
where $\left(\frac{\omega}{p}\right)=-1$ and $\big<*,*\big>$ denotes a
diagonal matrix; 
we write $V\simeq\mathbb H$ or $V\simeq\A$
accordingly.  Note that when $V\simeq\mathbb H$, $V$ contains 2 isotropic lines, and
when $V\simeq\A$, $V$ contains no isotropic lines.  Now suppose $p=2$; then either
$Q$ is a $GL_2(\F)$ conjugate of $I$ or (over $\F$) $Q=\left(\begin{array}{cc} 0&1\\1&0\end{array}\right)$
(which is stabilised under conjugation by $GL_2(\F)$).
When $V\simeq I$, $V$ contains 1 isotropic line; when $V\simeq\left(\begin{array}{cc} 0&1\\1&0\end{array}\right)$,
all 3 lines in $V$ are isotropic.

\section{Action of Hecke operators on Eisenstein series of square-free level}
Throughout this section, we assume $k\in\Z_+$ with $k\ge 4$.

For $\stufe$ square-free, it is not hard to show that the $GL_2(\Z)$-equivalence classes of
two coprime symmetric pairs $(M\ N)$, $(M'\ N')$ are in the same $\Gamma_0(\stufe)$-orbit
if and only if $\rank_qM=\rank_qM'$ for all primes $q|\stufe$.  Thus we can identify
each $\Gamma_0(\stufe)$ cusp with a (multiplicative) 
partition $\rho=(\stufe_0,\stufe_1,\stufe_2)$  where
$\stufe_0\stufe_1\stufe_2=\stufe$:  Take $M_{\rho}$ to be a diagonal matrix so that for each
prime $q|\stufe$, modulo $q$ we have
$$M_{\rho}\equiv\begin{cases} 0&\text{if $q|\stufe_0$,}\\
\left(\begin{array}{cc} 1\\&0\end{array}\right)&\text{if $q|\stufe_1$,}\\
I&\text{if $q|\stufe_2$.}
\end{cases}$$
Then $\rho$ corresponds to the cusp $\Gamma_{\infty}\gamma_{\rho}\Gamma_0(\stufe)$
where $\gamma_{\rho}=\left(\begin{array}{cc} I&0\\M_{\rho}&I\end{array}\right)$.
Note that
$$GL_2(\Z)(M_{\rho}\ I)
=SL_2(\Z)(M_{\rho}\ I)\cup SL_2(\Z)(M_{\rho}\ I)\left(\begin{array}{cccc} -1\\&1\\&&-1\\&&&1\end{array}\right),$$
so we can identify the cusp with $SL_2(\Z)(M_{\rho}\ I)\Gamma_0(\stufe)$.
We use $\E_{\rho}$ to denote $\E_{\gamma_{\rho}}$.  To ease
the discussions during our computations
we consider $2\E_{\rho}$ to be supported on
a set of representatives for the
$SL_2(\Z)$-equivalence classes in the $\Gamma_0(\stufe)$-orbit of $(M_{\rho}\ I)$.

For $q$ prime, we say $(M\ N)$ has $q$-type $i$ if $(M\ N)$ is a coprime symmetric pair with
$\rank_qM=i$.  
For $(M\ N)$ of $q$-type $i$, choose $E\in GL_2(\Z)$ so that $q$ divides the lower
$2-i$ rows of $EM$; then we say $(M\ N)$ (or simply $M$) is $q^2$-type $i,j$
where $j=\rank_q\left(\begin{array}{cc} I_i\\&\frac{1}{q}I_{2-i}\end{array}\right) EM$.
Given square-free $\stufe$ and a partition $\rho=(\stufe_0,\stufe_1,\stufe_2)$ of 
$\stufe$, we say $(M\ N)$ has $\stufe$-type $\rho$ if
$(M\ N)$ is a coprime symmetric pair and, for each prime $q|\stufe_i$,
$\rank_qM=i$.

Given a character $\chi$ modulo $\stufe$, and $(M\ N)=(M_{\rho}\ I)\gamma$ where
$\gamma=\left(\begin{array}{cc} A&B\\C&D\end{array}\right)\in\Gamma_0(\stufe)$, we can describe $\chi(\det D)$
in terms of $M,N,\rho$ as follows.  
For each prime $q|\stufe_0$, we have $N\equiv D\ (q)$, so $\chi_q(\det D)=\chi_q(\det N)$.
For each prime $q|\stufe_2$, we have $M\equiv A\equiv\,^t\overline D\ (q)$,
so $\chi_q(\det D)=\overline\chi_q(\det M)$.  Now take a prime $q|\stufe_1$; write
$D=\left(\begin{array}{cc} d_1&d_2\\d_3&d_4\end{array}\right)$.  Thus modulo $q$ we have
$$A\equiv\overline{\det D}
\left(\begin{array}{cc} d_4&-d_3\\-d_2&d_1\end{array}\right),$$
so consequently modulo $q$ we have
$$ M\equiv\overline{\det D}
\left(\begin{array}{cc} d_4&-d_3\\0&0\end{array}\right),\ N\equiv\left(\begin{array}{cc} *&*\\d_3&d_4\end{array}\right).$$
We know $q\nmid(d_3\ d_4)$, so $\chi_q(\det D)=\overline\chi_q(m_1)\chi_q(n_4)$ or
$\overline\chi_q(-m_2)\chi_q(n_3)$, whichever is non-zero.  Take $E\in SL_2(\Z)$
so that $q$ divides row 2 of $EM$;
thus $E\equiv\left(\begin{array}{cc}\alpha&\beta\\0&\overline\alpha\end{array}\right)\ (q)$,
and with 
$$EM=\left(\begin{array}{cc} m_1'&m_2'\\qm_3'&qm_4'\end{array}\right),\ 
EN=\left(\begin{array}{cc} n_1'&n_2'\\n_3'&n_4'\end{array}\right),$$ we have
$$\overline\chi_q(m_1')\chi_q(n_4')=\overline\chi_q(m_1)\chi_q(n_4),\ 
\overline\chi_q(-m_2')\chi_q(n_3')=\overline\chi_q(-m_2)\chi_q(n_3)$$
provided $\chi_q^2=1$.  So when $\chi_q^2=1$ and $(M\ N)$ has $q$-type 1,
we can choose $E\in SL_2(\Z)$ so that $EM=\left(\begin{array}{cc} m_1&m_2\\qm_3&qm_4\end{array}\right)$,
$EN=\left(\begin{array}{cc} n_1&n_2\\n_3&n_4\end{array}\right)$; set
$\chi_{(1,q,1)}(M,N)=\chi_q(m_1)\chi_q(n_4)$ or
$\chi_q(-m_2)\chi_q(n_3)$ (whichever is non-zero).
Then set
$$\chi_{\rho}(M,N)
=\prod_{q|\stufe_0}\chi_q(\det N)\prod_{q|\stufe_1}\chi_{(1,q,1)}(M,N)\prod_{q|\stufe_2}\overline\chi_q(\det M).$$
Hence $2\E_{\rho}(\tau)=\sum\overline\chi_{\rho}(M,N)\det(M\tau+N)^{-k}$ where 
$(M\ N)$ varies over a set of $SL_2(\Z)$-equivalence class representatives for pairs
of $\stufe$-type $\rho$.
Also note that for $G\in SL_2(\Z)$, $\chi_{\rho}(GM,GN)=\chi_{\rho}(M,N)$, and since
$\left(\begin{array}{cc} G\\&^tG^{-1}\end{array}\right)\in Sp_2(\Z)$, 
we have $\chi_{\rho}(MG,N\,^tG^{-1})=\chi_{\rho}(M,N)$.

For the remainder of this section, fix a partition $\rho=(\stufe_0,\stufe_1,\stufe_2)$ 
of $\stufe$, and fix a character
$\chi$ modulo $\stufe$.  We decompose $\chi$ as $\chi_{\stufe_0}\chi_{\stufe_1}\chi_{\stufe_2}$
where $\chi_{\stufe_i}$ has modulus $\stufe_i$; we assume $\chi$ has been chosen so that
$\chi_{\stufe_1}^2=1$.
For $p$ prime, let $\G_1(p)$ be a set of representatives for 
$$\left\{\gamma\in GL_2(\Z):\ \gamma\equiv\left(\begin{array}{cc} *&0\\ *&*\end{array}\right)\ (q)\ \right\}
\backslash SL_2(\Z);$$ note that for $p\nmid \stufe$, we can take the elements in
$\G_1(p)$ to be congruent modulo $\stufe$ to $I$.

When evaluating the action of the Hecke operators,
we often use the following simple lemmas.

\begin{lem}\label{lemma 3.1} Say $p$ is a prime and $(M\ N)$ is $p$-type 1,
$M=\left(\begin{array}{cc} m_1&m_2\\pm_3&pm_4\end{array}\right)$, 
and $N=\left(\begin{array}{cc} n_1&n_2\\n_3&n_4\end{array}\right)$.
Then $p|m_1$ if and only if $p|n_4$, and $p|m_2$ if and only if $p|n_3$.
\end{lem}

\begin{proof}  Say $p|m_1$.  Then $p\nmid m_2$ since $\rank_pM=1$; by symmetry,
$p|m_2n_4$ and hence $p|n_4$.  The other arguments needed to prove the lemma are
essentially identical to this. 
\end{proof}

\begin{lem}\label{lemma 3.2} Let $V=\F x_1\oplus\F x_2$ where $\F=\Z/p\Z$,
$p$ prime. For $G\in \G_1(p)$, let $(x'_1\ x'_2)=(x_1\ x_2)\,^tG$.
Then as $G$ varies over $\G_1(p)$, $\F x'_1$ varies over all lines in $V$.
\end{lem}

\begin{proof} 
Representatives for $\G_1(p)$ are $\left(\begin{array}{cc} 0&-1\\1&0\end{array}\right),$
$\left(\begin{array}{cc} 1&\alpha\\0&1\end{array}\right)$ where $\alpha$ varies modulo $p$.  Thus $\F x'_1$
varies as claimed.

\end{proof}

\begin{prop}\label{proposition 3.3}
For $p$ a prime not dividing $\stufe$,
we have
$$\E_{\rho}|T(p)
=(\chi_{\stufe_0}(p^2)\chi_{\stufe_1}(p)p^{2k-3}
+\chi_{\stufe_0\stufe_2}(p)p^{k-2}(p+1)
+\chi_{\stufe_1}(p)\chi_{\stufe_2}(p^2))
\E_{\rho}.$$
\end{prop}

\begin{proof}
Decompose $\E_{\rho}$ as $\E_0+\E_1+\E_2$ where $\E_i$ is supported on pairs of $\stufe$-type $\rho$
and $p$-type $i$.

Using the matrices for $T(p)$ as described in Propostion 3.1 of [4], we can describe
$\E_{\rho}|T(p)$ as follows.  First,
let $\G_1=\G_1(p).$
Then
\begin{align*}
&2\E_{\rho}(\tau)|T(p)=\\
& p^{2k-3}\sum_{{0\le r\le 2}\atop{G_r,Y_r}}\chi(p^{2-r})\overline\chi_{\rho}(M,N)
\det(MD_rG_r^{-1}\tau+pND_r^{-1}\,^tG_r+MY_r\,^tG_r)^{-k}
\end{align*}
where $(M\ N)$ varies over a set of $SL_2(\Z)$-equivalence class representatives
for pairs of $\stufe$-type $\rho$,
$D_r=\left(\begin{array}{cc} I_r\\&pI_{2-r}\end{array}\right)$,
$G_0=G_2=I$, $G_1$ varies over $\G_1$, $Y_0=0$,
$Y_1=\left(\begin{array}{cc} y\\&0\end{array}\right)$ with $y$ varying modulo $p$,
and $Y_2\in\Z^{2,2}_{\sym}$ varying modulo $p$.
For convenience, we choose $Y_i\equiv0\ (\stufe).$

{\bf Case I:} Say $r=0$.  We take $(M'\ N')=D_{\ell}^{-1}(pM\ N)$ where
$\ell=\rank_pN$, and the $SL_2(\Z)$-equivalence class representative $(M\ N)$
is chosen so that $p$ divides the lower $2-\ell$ rows of $N$.

First suppose $\ell=2$.  Thus $\rank_pM'=0$, and 
$$\chi(p^2)\chi_{\rho}(M',N')=\chi_{\stufe_0}(p^2)\chi_{\stufe_1}(p)\chi_{\rho}(M,N).$$
So the contribution from these terms is $\chi_{\stufe_0}(p^2)\chi_{\stufe_1}(p) p^{2k-3}\E_0$.

Next suppose $\ell=1$.  So $p$ does not divide row 2 of $M$, and hence
$\rank_pM'=1$ with $p$ dividing row 1 of $M'$, $p$ not dividing row 1 of $N'$
(and so $(M',N')=1$).  We have
$\chi(p^2)\chi_{\rho}(M',N')=\chi_{\stufe_0\stufe_2}(p)\chi_{\rho}(M,N).$
Reversing, take $(M'\ N')$ of $\stufe$-type $\rho$, $p$-type 1.  We need to count the
equivalence classes $SL_2(\Z)(M\ N)$ so that
$\left(\begin{array}{cc} 1\\&\frac{1}{p}\end{array}\right)(pM\ N)\in SL_2(\Z)(M'\ N')$.
For any $E\in SL_2(\Z)$, we have $\left(\begin{array}{cc} 1\\&p\end{array}\right) E
\left(\begin{array}{cc} 1\\&\frac{1}{p}\end{array}\right) \in SL_2(\Z)$ if and only if
$E\equiv\left(\begin{array}{cc} *&0\\*&*\end{array}\right)\ (p)$; thus we need to count the integral, coprime
pairs $(M\ N)=\left(\begin{array}{cc} 1\\&p\end{array}\right) E(M'/p\ \ N')$ where $E$ varies over $\G_1$.
(Note that we automatically have $M\,^tN$ symmetric since $M'\,^tN'$ is symmetric.)
We can assume $p$ divides row 2 of $M'$.  To have $M$ integral, we need $p$ dividing
row 1 of $EM'$; there is 1 choice of $E$ so that this is the case.
Thus $p$ does not divide row 2 of $M$ or row 1 of $N$, so $(M,N)=1$.  So the
contribution from these terms is $\chi_{\stufe_0\stufe_2}(p)p^{k-3}\E_1$.

Now suppose $\ell=0$.  Thus $\rank_pM=2=\rank_pM'$.  We have
$$\chi(p^2)\chi_{\rho}(M',N')=\chi_{\stufe_1}(p)\chi_{\stufe_2}(p^2)\chi_{\rho}(M,N).$$
So the contribution from these terms is $\chi_{\stufe_1}(p)\chi_{\stufe_2}(p^2)p^{-3}\E_2$.

{\bf Case II:} $r=1$.  Here we take 
$$(M'G\ N'\,^tG^{-1})=D_{\ell}^{-1}\left(M\left(\begin{array}{cc} 1\\&p\end{array}\right)\ 
N\left(\begin{array}{cc} p\\&1\end{array}\right)+MY\right),$$
where $\ell=\rank_p\left(M\left(\begin{array}{cc} 1\\&p\end{array}\right)\ N\left(\begin{array}{cc} p\\&1\end{array}\right)\right),$
$G\in\G_1$, $Y=\left(\begin{array}{cc} y\\&0\end{array}\right)$,
and the equivalence class representative $(M\ N)$ is adjusted so that $(M'\ N')$
is integral.

Suppose $\ell=2$.  Then $(M',N')=1$, $\rank_pM'=1$, and
$\chi(p)\chi_{\rho}(M',N')=\chi_{\stufe_0}(p^2)\chi_{\stufe_1}(p)
\chi_{\rho}(M,N).$
Reversing,
$$(M\ N)=\left(M'G\left(\begin{array}{cc} 1\\&\frac{1}{p}\end{array}\right)\ \left(N'\,^tG^{-1}-M'GY\right)
\left(\begin{array}{cc} \frac{1}{p}\\&1\end{array}\right)\right).$$
We can assume $p$ divides row 2 of $M'$; to have $M$ integral, we need to choose the unique
$G$ so that $q|m_2$ where $M'G=\left(\begin{array}{cc} m_1&m_2\\pm_3&pm_4\end{array}\right)$.  Then $p\nmid m_1$,
and by symmetry, $p\nmid n_4$ where $N'\,^tG^{-1}=\left(\begin{array}{cc} n_1&n_2\\n_3&n_4\end{array}\right)$.
To have $N$ integral, we need to choose the unique $y$ modulo $p$ so that
$n_1\equiv m_1y\ (p)$.  So $M,N$ are integral and coprime, and the contribution from
these terms is $\chi_{\stufe_0}(p^2)\chi_{\stufe_1}(p)p^{2k-3}\E_1$.

Say $\ell=1$.  Then we must have $M=\left(\begin{array}{cc} m_1&m_2\\pm_3&m_4\end{array}\right)$,
$N=\left(\begin{array}{cc} n_1&n_2\\n_3&pn_4\end{array}\right)$ with $p\nmid(m_1\ n_2)$; since $(M,N)=1$,
we must also have $p\nmid(m_4\ n_3)$.  Thus $(M',N')=1$ with $\rank_pM'=0,1,$ or 2,
and $$\chi(p)\chi_{\rho}(M',N')=\chi_{\stufe_0\stufe_2}(p)\chi_{\rho}(M,N).$$ 
Reversing,
\begin{align*}
(M\ N)=\\
&\left(\begin{array}{cc} 1\\&p\end{array}\right) E
\left(M'G\left(\begin{array}{cc} 1\\&\frac{1}{p}\end{array}\right)\ 
(N'\,^tG^{-1}-M'GY)\left(\begin{array}{cc} \frac{1}{p}\\&1\end{array}\right)\right),
\end{align*}
$E\in\G_1$.

Still assuming $\ell=1$, suppose $\rank_pM'=0$.  Then to have $N$ integral, for each $E$ we 
must choose the unique $G$ so that $q|n_1$ where $EN'\,^tG^{-1}=\left(\begin{array}{cc} n_1&n_2\\n_3&n_4\end{array}\right)$.
Thus $p\nmid n_2n_3$, and hence $\rank_pN=2$ for all choices of $y$.  So the contribution from these
terms is $\chi_{\stufe_0\stufe_2}(p)p^{k-3}\cdot p(p+1)\E_0$.

Continuing with the assumption $\ell=1$, suppose $\rank_pM'=1$; we can assume $p$ divides row 2 of $M'$.
To have $M,N$ integral, we need $p|m_2$, $n_1\equiv m_1y\ (p)$ where
$EM'G=\left(\begin{array}{cc} m_1&m_2\\m_3&m_4\end{array}\right)$, $EN'\,^tG^{-1}=\left(\begin{array}{cc} n_1&n_2\\n_3&n_4\end{array}\right)$.
Say $p$ divides row 2 of $EM'$ (this is the case for $q$ choices of $E$); then we need to choose the
unique $G$ so that $p|m_2$.  Then $p\nmid m_1$, and by symmetry, $p|n_3$.  But then $p$ divides
row 2 of both $M$ and $N$, so $(M,N)\not=1$.  So take the unique $E$ so that $p$ does not divide
row 2 of $EM'$; thus, by our choice of representatives in $\G_1$, we have $p$ dividing row 1 of $EM'$.
To have $N$ integral, we need to choose the unique $G$ so that $p|n_1$.  Then $p\nmid n_2$, and by
symmetry, $p|m_4$.  Since $\rank_pM'=1$, $p\nmid m_3$.  To have $(M,N)=1$, we need to choose the unique
$y$ so that $n_3\not\equiv m_3y\ (p)$; so we have $p-1$ choices for $y$.  The contributions from
these terms is $\chi_{\stufe_0\stufe_2}(p)p^{k-3}(p-1)\E_1$.

Now assume $\ell=1$, $\rank_pM'=2$.  Then for each $E$, choose the unique $G$ so that $p|m_2$
where $EM'G=\left(\begin{array}{cc} m_1&m_2\\m_3&m_4\end{array}\right)$; then $M$ is integral.  Also, $p\nmid m_1m_3$ so
$\rank_pM=2$.  Choose the unique $y$ so that $n_1\equiv m_1y\ (p)$; so $M,N$ are integral and coprime.
The contributions from these terms is $\chi_{\stufe_0\stufe_2}(p)p^{k-3}(p+1)\E_2$.

Now assume $\ell=0$; since $(M,N)=1$, we must have $M=\left(\begin{array}{cc} pm_1&m_2\\pm_3&m_4\end{array}\right)$,
$N=\left(\begin{array}{cc} n_1&pn_2\\n_3&pn_4\end{array}\right)$ with $p\nmid m_2n_3-n_1m_4$.  Thus $(M',N')=1$, with
$\rank_pM'\ge 1$, and
$$\chi(p)\chi_{\rho}(M',N')=\chi_{\stufe_1}(p)\chi_{\stufe_2}(p^2)\chi_{\rho}(M,N).$$  Reversing,
\begin{align*}
&(M\ N)=\\
&p\left(M'G\left(\begin{array}{cc} 1\\&\frac{1}{p}\end{array}\right)\ 
(N'\,^tG^{-1}-M'GY)\left(\begin{array}{cc} \frac{1}{p}\\&1\end{array}\right)\right).
\end{align*}

Say $\rank_pM'=1$; we can assume $p$ divides row 2 of $M'$.  So to have $(M,N)=1$ we need to choose
$G$ so that $p\nmid m_2$ where 
$$M'G=\left(\begin{array}{cc} m_1&m_2\\pm_3&pm_4\end{array}\right);$$ 
this gives us $p$ choices
for $G$.  Write $$N'\,^tG^{-1}=\left(\begin{array}{cc} n_1n_2\\n_3&n_4\end{array}\right);$$ 
by symmetry, if $p|n_3$ then we must
have $p|n_4$, and consequently $(M',N')\not=1$.  So $p\nmid n_3$, and $(M,N)=1$ for all choices
of $y$.  The contribution from these terms is $\chi_{\stufe_1}(p)\chi_{\stufe_2}(p^2)p^{-1}\E_1$.

Now say $\rank_pM'=2$; write 
$$M'G=\left(\begin{array}{cc} m_1&m_2\\m_3&m_4\end{array}\right),\  
N'\,^tG^{-1}=\left(\begin{array}{cc} n_1&n_2\\n_3&n_4\end{array}\right).$$
Then $(M,N)=1$ unless
$\left(\begin{array}{cc} n_1\\n_3\end{array}\right) - \left(\begin{array}{cc} m_1\\m_3\end{array}\right)\in\spn_p\left(\begin{array}{cc} m_2\\m_4\end{array}\right).$
This gives us $p-1$ choices for $y$, and the contribution is 
$\chi_{\stufe_1}(p)\chi_{\stufe_2}(p^2)p^{-3}(p^2-1)\E_2$.

{\bf Case III:} $r=2$.  So $(M'\ N')=D_{\ell}^{-1}(M\ pN+MY)$ where $Y\in\Z^{2,2}_{\sym}$ varies
modulo $p$, and $\ell=\rank_pM$ and we assume $p$ divides the lower $2-\ell$ rows of $M$.

Suppose $\ell=2$.  So $(M\ N)=(M'\ (N'-M'Y)/p)$; there is a unique $Y$ so that
$N'\equiv M'Y\ (p)$.  The contribution is 
$\chi_{\stufe_0}(p^2)\chi_{\stufe_1}(p)p^{2k-3}\E_2$.

Now suppose $\ell=1$ and $\rank_pM=1$; so $\rank_pM'\ge 1$, and 
$\chi_{\rho}(M',N')=\chi_{\stufe_0\stufe_2}(p)\chi_{\rho}(M,N).$
Reversing,
$$(M\ N)=\left(\begin{array}{cc} 1\\&p\end{array}\right) E(M'\ (N'-M'Y)/p)$$
where $E$ varies over $\G_1$.  Say $\rank_pM'=2$; write 
$$EM'G=\left(\begin{array}{cc} m_1&m_2\\m_3&m_4\end{array}\right),\ 
EN'\,^tG^{-1}=\left(\begin{array}{cc} n_1&n_2\\n_3&n_4\end{array}\right).$$  To have $N$ integral, we need
$(n_1\ n_2)\equiv(m_1\ m_2)Y\ (p)$, and to have $(M,N)=1$, we need
$(n_3\ n_4)\not\equiv (m_3\ m_4)Y\ (p)$.  Thus we have $p-1$ choices for $Y$, and the
contribution from these terms is $\chi_{\stufe_0\stufe_2}(p)p^{k-3}(p^2-1)\E_2$.

Say $\ell=1$ and $\rank_pM'=1$; assume $p$ divides row 2 of $M'$
(so $p$ does not divide row 2 of $N'$).
To have $N$ integral, we need $p$ not dividing row 1 of $EM'$
and $(n_1\ n_2)\equiv(m_1\ m_2)Y\ (p)$; this gives us $p$ choices
for $E$ and $p$ choices for $Y$.  Then $p$ does not divide row 1 of $M$ or row 2 of $N$,
so $(M,N)=1$.  The contribution from these terms is $\chi_{\stufe_0\stufe_2}(p)p^{k-3}\cdot p^2\E_1$.

Say $\ell=0$. So $\rank_pM'=0,1,$ or 2,
and $$\chi_{\rho}(M',N')=\chi_{\stufe_1}(p)\chi_{\stufe_2}(p^2)\chi_{\rho}(M,N).$$  Reversing, we have
$$(M\ N)=p(M'\ (N'-M'Y)/p).$$

Say $\rank_pM'=2$.  We need to choose $Y$ so that $\rank_p(\overline M'N'-Y)=2$; as $Y$ varies
over $\Z^{2,2}_{\sym}$ modulo $p$, so does $\overline M'N'-Y$, and $p^2(p-1)$ of these matrices
have $p$-rank 2.  Thus the contribution from these terms is
$\chi_{\stufe_1}(p)\chi_{\stufe_2}(p^2)p^{-1}(p-1)\E_2$.

Now say $\ell=0$ and $\rank_pM'=1$; we can assume $p$ divides row 2 of $M'$.  To have
$(M,N)=1$, we need $(n_1\ n_2)-(m_1\ m_2)Y\not\in\spn_p(n_3\ n_4)$.  For each $\alpha$
modulo $p$, we have $p$ choices of $Y$ so that
$(n_1\ n_2)-(m_1\ m_2)Y\equiv\alpha(n_3\ n_4)\ (p)$.  Thus we have $p^2(p-1)$ choices for
$Y$ so that $(M,N)=1$.  The contribution from these terms is
$\chi_{\stufe_1}(p)\chi_{\stufe_2}(p^2)p^{-1}(p-1)\E_1$.

Finally, say $\ell=0$ and $\rank_pM'=0$.  Thus $\rank_pN'=2=\rank_pN$ for all choices of $Y$.
So the contribution from these terms is $\chi_{\stufe_1}(p)\chi_{\stufe_2}(p^2)\E_0$.

Combining all the terms yields the result. 
\end{proof}

\begin{prop}\label{proposition 3.4} 
For $p$ a prime not dividing $\stufe$,
$$\E_{\rho}|T_1(p^2)
=(p+1)\big(\chi_{\stufe_0}(p^2)p^{2k-3}+\chi(p)p^{k-3}(p-1)+\chi_{\stufe_2}(p^2)\big)\E_{\rho}.$$
\end{prop}

\begin{proof} Let $\G_1=\G_1(p)$, and
let $\E_2, \E_{1}, \E_{2}$ be defined as in the proof of Proposition 3.3.
Decompose $\E_1,\E_2$ as $\E_1=\E_{1,1}+\E_{1,2}$, $\E_0=\E_{0,0}+\E_{0,1}+\E_{0,2}$ where
$\E_{i,j}$ is supported on pairs $(M\ N)$ of 
$p^2$-type $i,j$.  Further, we split $\E_{0,2}$ as $\E_{0,2,+}+\E_{0,2,-}$
where, for $p$ odd, $\E_{0,2,\nu}$ is supported on pairs $(M\ N)$ so that
$$\nu\left(\frac{\det(MN/p)}{p}\right)=1,$$ and for $p=2$, $\E_{0,2,+}$ is supported
on pairs $(M\ N)$ where $\frac{1}{2}M\,^tN\simeq I\ (2)$ and $\E_{0,2,-}$
on pairs $(M\ N)$ where $\frac{1}{2}M\,^tN\equiv\left(\begin{array}{cc} 0&1\\1&0\end{array}\right)\ (2)$.
When $p$ is odd, set $\epsilon=\left(\frac{-1}{p}\right)$.

Using Proposition 2.1 of [4], the action of $T_1(p^2)$ is given by
matrices parameterised by $r_0,r_1,r_2\in\Z_{\ge 0}$ where $r_0+r_1+r_2=1$.

{\bf Case I: $r_0=1$.}
Here the summands are 
\begin{align*}
&p^{k-3}\chi_{\rho}(M,N)\cdot\\
&\ \det\left(M\left(\begin{array}{cc} \frac{1}{p}\\&1\end{array}\right) G^{-1}\tau+N\left(\begin{array}{cc} p\\&1\end{array}\right)\,^tG+
M\left(\begin{array}{cc} \frac{1}{p}\\&1\end{array}\right) Y\,^tG\right)^{-k}
\end{align*}
where $(M\ N)$ varies over pairs of $\stufe$-type $\rho$,
$G$ varies over $\G_1$, and $Y=\left(\begin{array}{cc} y_1&y_2\\y_2&0\end{array}\right)$ with $y_1$
varying modulo $p^2$ and $y_2$ varying modulo $p$.  

{\bf Case I a:} Say $M\left(\begin{array}{cc} \frac{1}{p}\\&1\end{array}\right)$ is not integral.  Then we can adjust
the equivalence class representative for $(M'\ N')$ so that
\begin{align*}
&(M'G\ N'\,^tG^{-1})=\\
&\ \left(\begin{array}{cc} p\\&1\end{array}\right)\left( M\left(\begin{array}{cc} \frac{1}{p}\\&1\end{array}\right)
\ \ N\left(\begin{array}{cc} p\\&1\end{array}\right) +M\left(\begin{array}{cc} \frac{1}{p}\\&1\end{array}\right) Y\right)
\end{align*}
is coprime with $\rank_pM'=1$ or 2.  So
$\chi_{\rho}(M',N')=\chi_{\stufe_0}(p^2)\chi_{\rho}(M,N).$
Using the techniques used in the proof of Proposition 3.2, we find that in the
case $\rank_pM'=2$ the contribution is
$\chi_{\stufe_0}(p^2)p^{k-3}p^k(p+1)\E_2$, and in the case $\rank_pM'=1$ the contribution is
$\chi_{\stufe_0}(p^2)p^{k-3}p^{k+1}\E_1$.

{\bf Case I b:} Suppose $M\left(\begin{array}{cc} \frac{1}{p}\\&1\end{array}\right)$ is integral
and that $(M',N')=1$ where
$$(M'G\ N'\,^tG^{-1})
=\left(M\left(\begin{array}{cc} \frac{1}{p}\\&1\end{array}\right)\ \ N\left(\begin{array}{cc} p\\&1\end{array}\right)+
M\left(\begin{array}{cc} \frac{1}{p}\\&1\end{array}\right) Y\right).$$
Note that $\rank_pM'\ge 1$, else
$\rank_p N'\le 1$
and hence $(M',N')\not=1$.  Also, note that $\chi_{\rho}(M',N')
=\chi(p)\chi_{\rho}(M,N)$.

Reversing, when $\rank_pM'=2$, we find that for each $G$, there are $p-1$
choices for $Y$ so that $(M\ N)$ is an integral, coprime pair; the contribution
from these terms is  $\chi(p)p^{k-3}(p^2-1)\E_2$.
When $\rank_pM'=1$, we can assume $p$ divides row 2 of $M'$; to have $N$ integral
we need to choose the unique $G\in\G_1$ so that $p$ divides the $2,1$-entry
of $N'\,^tG^{-1}$, and then we have $p(p-1)$ choices for $Y$ so that
$(M,N)=1$.  So the contribution from these terms is
$\chi(p)p^{k-3}\cdot p(p-1)\E_1$.

{\bf Case Ic:} Suppose $M\left(\begin{array}{cc} \frac{1}{p}\\&1\end{array}\right)$ is integral
and that
$$\rank_p\left(M\left(\begin{array}{cc} \frac{1}{p}\\&1\end{array}\right)\ \ N\left(\begin{array}{cc} p\\&1\end{array}\right)\right)<2;$$
adjusting the equivalence class representative, we have $(M',N')=1$ where
\begin{align*}
&(M'G\ N'\,^tG^{-1})=\\
&\ \left(\begin{array}{cc} 1\\&\frac{1}{p}\end{array}\right)
\left(M\left(\begin{array}{cc} \frac{1}{p}\\&1\end{array}\right)\ \ N\left(\begin{array}{cc} p\\&1\end{array}\right)+M\left(\begin{array}{cc} \frac{1}{p}\\&1\end{array}\right)
Y\right).
\end{align*}  
Also, $\chi_{\rho}(M',N')=\chi_{\stufe_2}(p^2)\chi_{\rho}(M,N)$.

Reversing, take a coprime pair $(M'\ N')$.  Then arguing as above, we find that
when $\rank_pM'=2$, the contribution is 
$\chi_{\stufe_2}(p^2)p^{-2}(p^2-1)(p+1)\E_2$.
When $\rank_pM'=1$, the contribution is
$\chi_{\stufe_2}(p^2)p^{-1}(p^2+p-1)\E_1$.
When $\rank_pM'=0$, the contribution is
$\chi_{\stufe_2}(p^2)p^{k-3}\cdot p^{3-k}(p+1)\E_0.$

{\bf Case II: $r_1=1$.} Here the summands are
$$p^{k-3}\chi(p)\overline\chi_{\rho}(M,N)\det\left(MG^{-1}\tau+N\,^tG+M\left(\begin{array}{cc} \frac{y}{p}\\&0\end{array}\right)\,^tG\right)^{-k}$$
where $(M\ N)$ varies over pairs of $\stufe$-type $\rho$, 
 $y$ varies modulo $p$ with $p\nmid y$, $G$ varies over $^t\G_1$.

{\bf Case II a:} Say $M\left(\begin{array}{cc} \frac{1}{p}\\&1\end{array}\right)$ is not integral.  Adjust
the 
% $SL_2(\Z)$-equivalence class 
representative $(M\ N)$
$$(M'G\ N'\,^tG^{-1})=\left(\begin{array}{cc} p\\&1\end{array}\right) \left(M\ N+M\left(\begin{array}{cc} \frac{y}{p}\\&0\end{array}\right)\right)$$
is integral.  Then $(M',N')=1$ with $(M'\ N')$ of $p^2$-type $1,2$ or $0,1$ or $0,2$.
Also,
$\chi(p)\chi_{\rho}(M',N')=\chi_{\stufe_0}(p^2)\chi_{\rho}(M,N).$

Reversing, when $(M'\ N')$ is $p^2$-type $1,2$, we can assume $p$ divides row 2 of $M'$;
we find there are unique choices for $G,Y$ so that $N$ is integral, and then we
get $(M,N)=1$.  Thus the contribution from these terms is
$\chi_{\stufe_0}(p^2)p^{k-3}\cdot p^k\E_{1,2}.$

Next suppose $M'$ is $p^2$-type $0,1$; we can assume $p^2$ divides row 2 of $M'$.
For 1 choice of $E$ we have $p^2|(m_1\ m_2)$; but then we cannot have $N$ integral
and coprime to $M$.  For the other $p$ choices of $E$ we have $p^2|(m_3\ m_4)$.  Choose
the unique $G$ so that $p|n_2$; so $p\nmid n_1n_4$ and $p^2\nmid m_1$, else by
symmetry $p^2|m_2$, contradicting that $M'$ is $p^2$-type $(0,1)$.  Then choose the unique
$y\not\equiv0\ (p)$ so that $n_1\equiv m_1\frac{y}{p}\ (p)$; we get $p$ not dividing row 1 of $M$
or row 2 of $N$, so $(M,N)=1$.  The contribution from these terms is
$\chi_{\stufe_0}(p^2)p^{k-3}\cdot p^k\cdot p\E_{0,1}.$

Now suppose $(M'\ N')$ is $p^2$-type $0,2$.  For each choice of $E\in\G_1$,
we choose the unique $G\in\,^tG^{-1}$ so that $EN'G=\left(\begin{array}{cc} n_1&pn_2\\n_3&n_4\end{array}\right)$
(thus $p\nmid n_1n_4$). 
To have $N$ integral, we need to choose $y\not\equiv0\ (p)$
so that $n_1\equiv m_1\frac{y}{p}\ (p)$; this is possible if and only if $p^2\nmid m_1$.  
Let $V=\F x_1\oplus\F x_2$
be equipped with the quadratic form $\frac{1}{p}M'\,^tN'$ relative to the basis
$(x_1\ x_2)$.  Then with $(x_1'\ x_2')=(x_1\ x_2)\,^tE$, $\F x_1'$ varies
over all lines in $V$, and $\frac{1}{p}EM'\,^tN'\,^tE$ represents the quadratic form
relative to $(x_1'\ x_2')$.  When $p$ is odd and $V\simeq\mathbb H$, there are 2 choices of $E$ so that
$\F x_1'$ is isotropic; equivalently, when $V\simeq\mathbb H$ there are 2 choices of $E$ so
that $p^2|m_1$ (since the value of the quadratic form on $x_1'$ is $m_1n_\frac{1}{p}\in\F$).
When $p$ is odd and $V\not\simeq\mathbb H$, we have $p^2\nmid m_1$ for all $p+1$ choices of $E$.
So the contribution when $p$ is odd is 
$$\chi_{\stufe_0}(p^2)p^{k-3}\cdot p^k(p-1)\E_{0,2,\epsilon}
+\chi_{\stufe_0}(p^2)p^{k-3}\cdot p^k(p+1)\E_{0,2,-\epsilon}.$$
When $p=2$ and $V\simeq\left(\begin{array}{cc} 0&1\\1&0\end{array}\right)$
we have $\F x_1'$ isotropic for all $E$; when $V\simeq I$ we have $\F x_1'$ isotropic for
1 choice of $E$.  So the contribution when $p=2$ is
$2\chi_{\stufe_0}(p^2)p^{2k-3}\E_{0,2,+}$.

{\bf Case II b:} Suppose $M\left(\begin{array}{cc} \frac{1}{p}\\&1\end{array}\right)$ is integral, and
$$(M'G\ N'\,^tG^{-1})=\left(M\ N+M\left(\begin{array}{cc} \frac{y}{p}\\&0\end{array}\right)\right)$$
is a coprime symmetric pair with $\rank_pM'\le 1$.
Note that $\chi(p)\chi_{\rho}(M',N')=\chi(p)\chi_{\rho}(M,N)$.

Reversing, first suppose that $\rank_pM'=1$, and assume $p$ divides
row 2 of $M'$.
To have $N$ integral, 
choose the unique $G$ so that $p|m_1$
(so $p\nmid m_2$).  To have $(M,N)=1$, we need $p$ not dividing row 2 of $N$.
If $M'$ is $p^2$-type $(1,1)$, we can assume $p^2$ divides row 2 of $M'$ and then
for any choice of $y\not\equiv0\ (p)$ we have $p$ not dividing row 2 of $N$.
If $M'$ is $p^2$-type $(1,2)$ then $p^2\nmid m_3$ so there are $p-2$ choices for
$y\not\equiv0\ (p)$ so that $p$ does not divide row 2 of $N$.  Thus the contribution
from these terms is $\chi(p)p^{k-3}(p-1)\E_{1,1}+\chi(p)p^{k-3}(p-2)\E_{1,2}$.

Now suppose $M'$ is $p^2$-type $(0,0)$; then $N$ is invertible modulo $p$
for all choices of $G$ and $y$.  Hence the
contribution is $\chi(p)p^{k-3}(p^2-1)\E_{0,0}$.

Suppose $M'$ is $p^2$-type $(0,1)$; then we can assume $p^2$ divides row 2 of
$M'$.
For 1 choice of $G$ we have $p^2|m_1$
and then for any $y$ we have $\rank_pN=2$.
Say we take any of the other $p$ choices for $G$ so that 
$p\nmid m_1$.  By symmetry, $p\nmid n_4$.
So there are $p-2$ choices of $y\not\equiv0\ (p)$ so that $\rank_pN=2$.
Thus the contribution from these terms is $\chi(p)p^{k-3}(p^2-p-1)\E_{0,1}.$

Suppose $M'$ is $p^2$-type $(0,2)$.  Let $V=\F x_1\oplus\F x_2$ be
equipped with the quadratic form $\frac{1}{p}\overline N'M'$ relative to
$(x_1\ x_2)$.  Then with $(x_1'\ x_2')=(x_1\ x_2)G$, $\F x_1'$ varies over all
lines in $V$ as $G$ varies over $^t\G_1$, and the value of the quadratic form
on $x_1'$ is $\overline{\det N'}(m_1n_4-m_3n_2)/p$.  We have $\rank_2 N=2$ if and
only if $\det N'\not\equiv y(m_1n_4-m_3n_2)/p\ (p)$.  When 
$p$ is odd and $V\simeq\mathbb H$, there are
2 choices of $G$ so that $x_1'$ is isotropic, and then $\rank_pN=2$ for all
$y\not\equiv0\ (p)$; for a choice of $G$ so that $x_1'$ is anisotropic, there
are $p-2$ choices of $y\not\equiv0 (p)$ so that $\rank_pN=2$.  Hence the
contribution from these terms when $p$ is odd is
$\chi(p)p^{k-3}\cdot p(p-1)\E_{0,2,\epsilon}+\chi(p)p^{k-3}\cdot (p+1)(p-2)\E_{0,2,-\epsilon}$.
When $p=2$,
we have $y\equiv 1\ (p)$; when $V\simeq\left(\begin{array}{cc} 0&1\\1&0\end{array}\right)$,
$\F x_1'$ is isotropic for all 3 $G$, and 
when $V\simeq I$ there is 1 choice of $G$ so that $\F x_1'$ is isotropic.
So the contribution when $p=2$ is 
$\chi(p)p^{k-3}\E_{0,2,+}+3\chi(p)p^{k-3}\E_{0,2,-}$.

{\bf Case II c:} Suppose $M\left(\begin{array}{cc} \frac{1}{p}\\&1\end{array}\right)$ is integral and
$$\rank_p\left(M\ N+M\left(\begin{array}{cc} \frac{y}{p}\\&0\end{array}\right) \right)=1;$$
we can adjust the representative so that
$$(M'G\ N'\,^tG^{-1})=\left(\begin{array}{cc} 1\\&\frac{1}{p}\end{array}\right)
\left(M\ N+M\left(\begin{array}{cc} \frac{y}{p}\\&0\end{array}\right) \right)$$
is an integral pair.  Then $(M',N')=1$ with $\rank_pM'\ge 1$, and
$$\chi(p)\chi_{\rho}(M',N')=\chi_{\stufe_2}(p^2)\chi_{\rho}(M,N).$$

Reversing, we need to count the equivalence classes $SL_2(\Z)(M\ N)$ so that
$\left(\begin{array}{cc} 1\\&\frac{1}{p}\end{array}\right)\left(M\ N+M\left(\begin{array}{cc} \frac{y}{p}\\&0\end{array}\right)\right)
\in SL_2(\Z)(M'\ N').$  Thus we need to count the integral, coprime pairs
$$(M\ N)=\left(\begin{array}{cc} 1\\&p\end{array}\right) E\left(M'G\ N'\,^tG^{-1}
-M'G\left(\begin{array}{cc} \frac{y}{p}\\&0\end{array}\right)\right),$$
where $E\in\G_1$.  Write $EM'G=\left(\begin{array}{cc} m_1&m_2\\m_3&m_4\end{array}\right)$,
$EN'\,^tG^{-1}=\left(\begin{array}{cc} n_1&n_2\\n_3&n_4\end{array}\right)$.

Say $\rank_pM'=2$.  To have $N$ integral, choose the unique $G$ so that
$p|m_1$ (and hence $p\nmid m_2m_3$).  Then for any choice of $y$, $(M,N)=1$.
So the contribution from these terms is $\chi_{\stufe_2}(p^2)p^{k-3}\cdot p^{-k}(p^2-1)\E_2$.

Now say $\rank_pM'=1$; we can assume $p$ divides row 2 of $M'$.  For $p$
choices of $E$, we have $p$ dividing row 2 of $EM'$, and then $p$ divides row 2
of $(M\ N)$ (meaning $(M,N)\not=1$).  So take the unique $E$ so that $p$ divides
row 1 of $EM'$ (and hence $p$ does not divide row 1 of $EN'$).  So to have
$(M,N)=1$, we need to choose $G$ so that $p\nmid m_3$; this gives us $p$ choices
for $G$.  Then for every $y\not\equiv0\ (p)$, we have $(M,N)=1$.  So the contribution
from these terms is $\chi_{\stufe_2}(p^2)p^{k-3}\cdot p^{-k}\cdot p(p-1)\E_1$.

{\bf Case III: $r_2=1$.} Here the summands are 
$$\chi(p^2)\overline\chi_{\rho}(M,N)p^{k-3}\det\left(M\left(\begin{array}{cc} p\\&1\end{array}\right)
G^{-1}\tau+N\left(\begin{array}{cc} \frac{1}{p}\\&1\end{array}\right)\,^tG\right)^{-k},$$
where $(M\ N)$ varies over pairs of $\stufe$-type $\rho$, and $G$
varies over $^t\G_1$.

{\bf Case III a:} Suppose $N\left(\begin{array}{cc} \frac{1}{p}\\&1\end{array}\right)$ is not
integral.  We can adjust the representative so that $N=\left(\begin{array}{cc}
n_1&n_2\\pn_3&n_4\end{array}\right)$ (so $p\nmid n_1$).  Write $M=\left(\begin{array}{cc}
m_1&m_2\\m_3&m_4\end{array}\right)$.  Then $p\nmid (m_4\ n_4)$, else by
symmetry, $p|m_3$ and $(M,N)\not=1$.  Thus with 
\begin{align*}
&(M'G\ N'\,^tG^{-1})=\\
&\left(\begin{array}{cc} p\\&1\end{array}\right)\big(M\left(\begin{array}{cc}
p\\&1\end{array}\right)\ N\left(\begin{array}{cc} \frac{1}{p}\\&1\end{array}\right)\big),
\end{align*}
we have $(M',N')=1$ and $\rank_pM'\le 1$.
When $\rank_pM'=1$, we must have $M'$ of $p^2$-type
$1,1$ with $p$ dividing row 1 of $M'$, and
$$\chi(p^2)\chi_{\rho}(M',N')=\chi_{\stufe_0}(p^2)\chi_{\rho}(M,N).$$

Reversing, 
$$(M\ N)=\left(\begin{array}{cc} \frac{1}{p}\\&1\end{array}\right) E\left(M'G\left(\begin{array}{cc}
\frac{1}{p}\\&1\end{array}\right) \ \ N'\,^tG^{-1}\left(\begin{array}{cc} p\\&1\end{array}\right)\right).$$
We know $\left(\begin{array}{cc} \frac{1}{p}\\&1\end{array}\right) E\left(\begin{array}{cc} p\\&1\end{array}\right)\in
SL_2(\Z)$ if and only if $E\equiv\left(\begin{array}{cc} *&0\\*&*\end{array}\right)\ (p)$.
Write $EM'G=\left(\begin{array}{cc} m_1&m_2\\m_3&m_4\end{array}\right),$
$EN'\,^tG^{-1}=\left(\begin{array}{cc} n_1&n_2\\n_3&n_4\end{array}\right)$.

First suppose $M'$ is $p^2$-type $1,1$ with $p$ dividing row 1 of
$M'$.  We know $p$ divides row 1 of $EM'$ if and only if
$E\equiv\left(\begin{array}{cc} *&0\\*&*\end{array}\right)\ (p)$; thus we only need to
consider $E=I$, and we can assume $p^2$
divides row 1 of $M'$.  To ensure $N$ is integral, we need to choose
the unique $G\in\,^t\G_1$ so that $p|n_2$; then by symmetry, $p|n_3$, and
since $\rank_pM'=1$, $p\nmid m_4$.  Then $M,N$
are integral and coprime.
So the contribution from these terms is $\chi_{\stufe_0}(p^2)p^{k-3}\cdot p^k\E_{1,1}$.

Now suppose $M'$ is $p^2$-type $0,0$.  Then for each $E\in\G_1$, there
is a unique $G\in\,^t\G_1$ so that $p|n_2$ (and hence $p\nmid n_1n_4$).  
Thus $\rank_pN=1$ so $(M,N)=1$, and the contribution from these terms is
$\chi_{\stufe_0}(p^2)p^{k-3}\cdot p^k(p+1)\E_{0,0}$.

Suppose $M'$ is $p^2$-type $0,1$;  we can assume $p^2$
divides row 2 of $EM'$.  
Suppose $p^2$ divides row 2 of $EM'$; to have $M$ integral, we choose
the unique $G$ so that $p^2|m_1$.  Then $p^2\nmid m_2$, and by symmetry,
$p|n_4$.  Hence $p\nmid n_2n_3$, and $N$ is not integral.  So choose $E$
so that $p^2$ does not divide row 2 of $EM'$; we have 1 choice for $E$, and
then $p^2$ divides row 1 of $EM'$.  Choose the unique $G$ so that $p|n_2$; then
$p\nmid n_1n_4$, and $N$ is integral with $\rank_pN=2$.  So the contribution
from these terms is
$\chi_{\stufe_0}(p^2)p^{k-3}\cdot p^k\E_{0,1}$.

Suppose $M'$ is $p^2$-type $0,2$.  For each $E\in\G_1$, choose the unique
$G\in\,^t\G_1$ so that $q|n_2$.  To have $M$ integral, we need $p^2|m_1$.
Let $V=\F x_1\oplus\F x_2$ be equipped with the quadratic form given
by $\frac{1}{p}M'\,^tN'$ relative to $(x_1\ x_2)$.  Then with $(x_1'\
x_2')=(x_1\ x_2)\,^tE$, the quadratic form on $V$ is given by
$\frac{1}{p}EM'\,^tN'\,^tE$ relative to $(x_1'\ x_2')$.  As $E$ varies
over $\G_1$, $\F x_1'$ varies over all lines in $V$.  Suppose
$p$ is odd; then to have $p^2|m_1$,
we need $V\simeq\mathbb H$, and then $p|m_1$ for 2 choices of $E$.
Hence the contribution from these terms when $p$ is odd is
 $2\chi_{\stufe_0}(p^2)p^{k-3}\cdot p^k\E_{0,2,\epsilon}$.
In the case that $p=2$, the contribution is
$\chi_{\stufe_0}(p^2)p^{2k-3}\E_{0,2,+}+3\chi_{\stufe_0}(p^2)p^{2k-3}\E_{0,2,-}$.

{\bf Case III b:} Suppose $N\left(\begin{array}{cc} \frac{1}{p}\\&1\end{array}\right)$ is integral,
and $\rank_p(M'\ N')=2$ where 
$$(M'G\ N'\,^tG^{-1})=\left(M\left(\begin{array}{cc} p\\&1\end{array}\right)\ N\left(\begin{array}{cc}
\frac{1}{p}\\&1\end{array}\right)\right).$$ 
So $\rank_pM\ge 1$, $\rank_pM'\le 1$, and $M'$ cannot be $p^2$-type 
$0,0$.  Also, when $M'$ is $p$-type 1, we can assume $p$ divides
row 2 of $M\left(\begin{array}{cc} p\\&1\end{array}\right)$; then using symmetry, 
we see $p$ divides
row 2 of $N$, so we must have $M'$ of $p^2$-type $1,2$.  Note that
$\chi(p^2)\chi_{\rho}(M',N')=\chi(p)\chi_{\rho}(M,N).$

Reversing, $(M\ N)=\left(M'G\left(\begin{array}{cc} \frac{1}{p}\\&1\end{array}\right)\ \
N'\,^tG^{-1}\left(\begin{array}{cc} p\\&1\end{array}\right)\right).$
Write 
$$M'G=\left(\begin{array}{cc} m_1&m_2\\m_3&m_4\end{array}\right),\ 
N'\,^tG^{-1}=\left(\begin{array}{cc} n_1&n_2\\n_3&n_4\end{array}\right).$$

Suppose $M'$ is $p^2$ type $1,2$; assume $p$ divides row 2 of $M'$.
Choose the unique $G$ so that $p|m_1$; thus $p\nmid m_2$, $p^2\nmid m_3$.
So $(M,N)=1$ and the contribution from these terms is
$\chi(p)p^{k-3}\E_{1,2}.$

Suppose $M'$ is $p^2$-type $0,1$; assume $p^2$ divides row 2 of $M'$.
We have $(M,N)=1$ if and only if $p^2\nmid m_1n_4$; by symmetry,
$p^2|m_1$ if and only if $p|n_4$.  So choose $G$
so that $p|m_1$; we have $p$ choices for $G$, and hence the contribution from these terms is
$\chi(p)p^{k-3}\cdot p\E_{0,1}.$

Now suppose that $(M'\ N')$ is $p^2$-type $0,2$.  Let $V=\F x_1\oplus\F x_2$ be
equipped with the quadratic form given by $\frac{1}{p}\overline N'M'$
relative to $(x_1\ x_2)$.
So relative to $(x_1'\ x_2')=(x_1\ x_2)G$, the quadratic form is given by
$\frac{1}{p}\,^tG\overline N'M'G\equiv\overline d\left(\begin{array}{cc}
n_4m_\frac{1}{p}&*\\*&*\end{array}\right)\ (p)$ where $d=\det N'$ (recall that
$p^2|m_3$).  We know $\F x_1'$ varies over all lines in $V$ as $G$
varies over $^t\G_1$.  When $p$ is odd, $p^2\nmid m_1n_4$ for $p-1$ choices of $G$ if
$V\simeq\mathbb H$, and $p^2\nmid m_1n_4$ for $p+1$ choices of $G$ otherwise.
So the contribution from these terms when $p$ is odd is
$\chi(p)p^{k-3}(p-1)\E_{0,2,\epsilon}+\chi(p)p^{k-3}(p+1)\E_{0,2,-\epsilon}.$ 
When $p=2$ the contribution is $2\chi(p)p^{k-3}\E_{0,2,+}$.

{\bf Case III c:} Suppose $N\left(\begin{array}{cc} \frac{1}{p}\\&1\end{array}\right)$ is integral
with $$\rank_p\left(M\left(\begin{array}{cc} p\\&1\end{array}\right)\ N\left(\begin{array}{cc} \frac{1}{p}\\&1\end{array}\right)\right)=1.$$
Adjust the equivalence class representative of $(M\ N)$ so that $p$
divides row 2 of $\left(M\left(\begin{array}{cc} p\\&1\end{array}\right)\ N\left(\begin{array}{cc}
\frac{1}{p}\\&1\end{array}\right)\right)=1.$
Set
$$(M'G\ N'\,^tG^{-1})=\left(\begin{array}{cc} 1\\&\frac{1}{p}\end{array}\right)
\left(M\left(\begin{array}{cc} p\\&1\end{array}\right)\ N\left(\begin{array}{cc} \frac{1}{p}\\&1\end{array}\right)\right)=1.$$
Since $M',N'$ are integral and $(M,N)=1$, we have  $(M',N')=1$.  Also, 
$\rank_pM'\ge 1$ and
$\chi(p^2)\chi_{\rho}(M',N')=\chi_{\stufe_2}(p^2)\chi_{\rho}(M,N).$

Reversing, take $(M'\ N')$ of $p$-type 1 or 2, and set
$$(M\ N)=\left(\begin{array}{cc} 1\\&p\end{array}\right) E\left(M'G\left(\begin{array}{cc}
\frac{1}{p}\\&1\end{array}\right)\ \ N'\,^tG^{-1}\left(\begin{array}{cc} p\\&1\end{array}\right)\right)$$
where $E\in\G_1$.
Write $EM'G=\left(\begin{array}{cc} m_1&m_2\\m_3&m_4\end{array}\right)$,
$EN'\,^tG^{-1}=\left(\begin{array}{cc} n_1&n_2\\n_3&n_4\end{array}\right)$.

First suppose $\rank_pM'=2$.  For each $E$, choose the unique $G$ so that $p|m_1$.
Thus $p\nmid m_2m_3$, so $(M,N)=1$.  The contribution from these terms is
$\chi_{\stufe_2}(p^2)p^{k-3}\cdot p^{-k}\cdot (p+1)\E_2$.

Finally, suppose $\rank_pM'=1$; assume $p$ divides row 2 of $M'$.  Suppose $p$
divides row 2 of $EM'$; choosing $G$ so that $p|m_1$, by symmetry $p|n_4$ and hence
$(M,N)\not=1$.  So choose the unique $E$ so that $p$ does not divide row 2 of $EM'$; then $p$
divides row 1 of $EM'$.  To have $(M,N)=1$, choose $G$ so that $p\nmid m_3$; we have
$p$ choices for $G$, and the contribution from these terms is
$\chi_{\stufe_2}(p^2)p^{k-3}\cdot p^{1-k}\E_1$.

Combining all the contributions yields the result.

\end{proof}

Now we determine the action on Eisenstein series of $T(q)$, $T_1(q^2)$ where $q$ is a prime
dividing $\stufe$.
We let $\F$ denote $\Z/q\Z$; when $q$ is odd, we let $\epsilon=\left(\frac{-1}{q}\right)$, and
we fix $\omega$ so that $\left(\frac{\omega}{q}\right)=-1$.
Let $\G_1=\G_1(q)$; note that we can choose the elements of $\G_1$ to be congruent modulo
$\stufe/q$ to $I$.

\begin{prop}\label{proposition 3.5} Suppose $q$ is a prime dividing $\stufe_2$; then
$$\E_{\rho}|T(q)=\chi_{\stufe_0}(q^2)\chi_{\stufe_1}(q)q^{2k-3}\E_{\rho}.$$
\end{prop}

\begin{proof} 
From Proposition 3.1 of \cite{HW}, we know that 
\begin{align*}
&2\E_{\rho}(\tau)|T(q)=\\
&\ q^{2k-3}\sum\overline\chi_{(\stufe_0,\stufe_1,\stufe_2/q)}(M,N)
\chi_q(\det M)\det(M\tau+qN+MY)^{-k}
\end{align*}
where $(M\ N)$ varies over a set of $SL_2(\Z)$-equivalence class
representatives for pairs of $\stufe$-type $\rho$, and
$Y$ varies over all symmetric matrices modulo $q$; recall that we
can choose $Y\equiv0\ (\stufe/q)$.
Take $(M\ N)$ of $\stufe$-type ${\rho}$; set
$(M'\ N')=(M\ qN+MY)$.  Thus $(M'\ N')$ is $\stufe$-type $\rho$;
also, $\chi_{(\stufe_0,\stufe_1,\stufe_2/q)}(M',N')=
\chi_{(\stufe_0,\stufe_1,\stufe_2/q)}(M,N)$.  Reversing, take $(M'\ N')$ of $\stufe$-type $\rho$; set
$(M\ N)=\left(M'\ \frac{1}{q}(N'-M'Y)\right).$  So we need to choose
$Y\equiv\overline M'N'\ (q)$ to have $N$ integral.  
Thus $\overline\chi_{(\stufe_0,\stufe_1,\stufe_2/q)}(M',N')
=\overline\chi_{\stufe_0}(q^2)\chi_{\stufe_1}(q)\overline\chi_{(\stufe_0,\stufe_1,\stufe_2/q)}(M,N).$
\end{proof}

\begin{prop}\label{proposition 3.6} For $q$ a prime dividing $\stufe_1$, let $\rho'=(\stufe_0,\stufe_\frac{1}{q},q\stufe_2)$;
then
$$\E_{\rho}|T(q)=\begin{cases}
\chi_{\stufe_0\stufe_2}(q)(q^{k-1}\E_{\rho}+q^{k-3}(q^2-1)\E_{\rho'})&\text{if $\chi_q=1$,}\\
\chi_{\stufe_0\stufe_2}(q)q^{k-1}\E_{\rho}&\text{if $\chi_q\not=1$.}
\end{cases}$$
\end{prop}

\begin{proof}  Recall that we must have $\chi_q^2=1$ since $q|\stufe_1$.
Take $(M\ N)$ of $\stufe$-type $\rho$ so that $q$ divides row 2 of $M$.  Set
$$(M'\ N')=\left(\begin{array}{cc} 1\\&\frac{1}{q}\end{array}\right) (M\ qN+MY);$$
since $q$ does not divide row 1 of $M$ or row 2 of $N$, $(M',N')=1$.
Also, $\rank_qM'\ge 1$, and
$\chi_{(\stufe_0,\stufe_\frac{1}{q},\stufe_2)}(M',N')
=\chi_{\stufe_0\stufe_2}(q)\chi_{(\stufe_0,\stufe_\frac{1}{q},\stufe_2)}(M,N).$

Reversing, take $(M'\ N')$ of $q$-type 1 or 2, $\stufe/q$-type $(\stufe_0,\stufe_\frac{1}{q},\stufe_2)$; set
$$(M\ N)=\left(\begin{array}{cc} 1\\&q\end{array}\right) E(M'\ (N'-M'Y)/q),$$
$E\in\G_1$.  Write $EM'=\left(\begin{array}{cc} m_1&m_2\\m_3&m_4\end{array}\right)$,
$EN'=\left(\begin{array}{cc} n_1&n_2\\n_3&n_4\end{array}\right)$.

First suppose $\rank_qM'=2$.
To have $N$ integral, we need to choose $Y$ so that
$(n_1\ n_2)\equiv(m_1\ m_2)Y\ (q)$; then to have $(M,N)=1$, we need
$(n_3\ n_4)\not\equiv(m_3\ m_4)Y\ (q)$.  
If $q\nmid m_1$ then $y_4$ can be chosen freely; if $q|m_1$ then $y_1$ can be chosen
freely.  This gives us $q-1$ choices for $Y$.  
Summing over these $Y$, we have
$$\sum_Y\chi_{(1,q,1)}(M,N)=\begin{cases} q-1&\text{if $\chi_q=1$,}\\
0&\text{if $\chi_q\not=1$.}
\end{cases}$$
So the contribution from these terms is
$$\begin{cases} q^{2k-3}\cdot q^{-k}(q^2-1)\E_{\rho'}&\text{if $\chi_q=1$,}\\
0&\text{if $\chi_q\not=1$.}
\end{cases}$$

Now suppose $\rank_qM'=1$; assume $q$ divides row 2 of $M'$.  
To have $\rank_qM=1$,
we cannot have $q$ dividing row 1 of $EM'$; this leaves us $q$ choices for $E$, and
with these choices we have $q$ dividing row 2 of $EM'$.  Then choose $Y$ so that
$(n_1\ n_2)\equiv(m_1\ m_2)Y\ (q)$; we have $q$ choices for $Y$.  Then row 2 of $N$ is
congruent modulo $q$ to $(n_3\ n_4)$, so $(M,N)=1$ and
$$\sum_Y\chi_{(1,q,1)}(M,N)=(q-1)\chi_{(1,q,1)}(M',N').$$
So the contribution from these terms is $\chi_{\stufe_0\stufe_2}(q)q^{k-1}\E_{\rho}$. 
\end{proof}

\begin{prop}\label{proposition 3.7} For $q$ a prime dividing $\stufe_0$, 
set $$\rho'=(\stufe_0/q,q\stufe_1,\stufe_2),\ \rho''=(\stufe_0/q,\stufe_1,q\stufe_2).$$
Then
\begin{align*}
&\E_{\rho}|T(q)\\
&\quad =
{\begin{cases} \chi_{\stufe_1}(q)\chi_{\stufe_2}(q^2)(\E_{\rho}+q^{-1}(q-1)\E_{\rho'}\\
\quad +q^{-1}(q-1)\E_{\rho''})
&\text{if $\chi_q=1$,}\\
\chi_{\stufe_1}(q)\chi_{\stufe_2}(q^2)(\E_{\rho}+\epsilon q^{-2}(q-1)\E_{\rho''})
&\text{if $\chi_2^2=1$, $\chi_q\not=1$,}\\
\chi_{\stufe_1}(q)\chi_{\stufe_2}(q^2)\E_{\rho}&\text{if $\chi_q^2\not=1$.}
\end{cases}}
\end{align*}
\end{prop}

\begin{proof} Take $(M\ N)$ of $\stufe$-type $\rho$, and set
$$(M'\ N')=\frac{1}{q}(M\ qN+MY).$$
Since $\rank_qN=2$, we have $\rank_q(M'\ N')=\rank_q(M'\ N)=2$;
hence $(M',N')=1$.  Also, 
$$\chi_{(\stufe_0/q,\stufe_1,\stufe_2)}(M',N')
=\chi_{\stufe_1}(q)\chi_{\stufe_2}(q^2)\chi_{(\stufe_0/q,\stufe_1,\stufe_2)}(M,N).$$

Reversing, take $(M'\ N')$ of $\stufe/q$-type $(\stufe_0/q,\stufe_1,\stufe_2).$
Set $$(M\ N)=(qM'\ N'-M'Y).$$ Write $M'=\left(\begin{array}{cc} m_1&m_2\\m_3&m_4\end{array}\right)$,
$N'=\left(\begin{array}{cc} n_1&n_2\\n_3&n_4\end{array}\right)$.  We need to choose $Y$ so that $\rank_qN=2$.
Then we sum $\sum_Y\overline\chi_q(\det N)$; since $\chi_q(\det N)=0$ when $\rank_qN<2$,
we can simply sum over all $Y$.

First suppose $\rank_qM'=2$.  Then $\overline M'N'-Y$ varies over all symmetric matrices
modulo $q$ as $Y$ does (here $\overline M'M'\equiv I\ (q)$), and each invertible
symmetric matrix is either in the $|GL_2(\F)$-orbit of $I$ or of $J=\left(\begin{array}{cc} 1\\&\omega\end{array}\right)$,
where $\left(\frac{\omega}{q}\right)=-1$ and $GL_2(\F)$ acts by conjugation.  With $q$ odd and
$G$ varying over $GL_2(\F)$, we have
$$\sum_Y\overline\chi_q(\det N)
=\frac{\chi_q(\det M')}{o(I)}\sum_G\overline\chi_q^2(\det G)
+\frac{\chi_q(\omega\det M')}{o(J)}\sum_G\overline\chi_q^2(\det G)$$
where 
$o(T)=\#\{C\in GL_2(\F):\ ^tCTC=T\ \};$ it is known that $o(I)=2(q-\epsilon)$ and
$o(J)=2(q+\epsilon)$ when $q$ is odd.
Now, $\overline\chi_q^2\circ\det$ is a character on $GL_2(\F)$, and it is
the trivial character if and only if $\chi_q^2=1$.  Hence
$$\sum_Y\overline\chi_q(\det N)=
\begin{cases} q^2(q-1)&\text{if $\chi_q=1$,}\\
\epsilon q(q-1)\chi_q(\det M')&\text{if $\chi_q^2=1$, $\chi_q\not=1$,}\\
0&\text{if $\chi_q^2\not=1$.}
\end{cases}$$
So the contribution from these terms when $q$ is odd is
$$\begin{cases} \chi_{\stufe_1}(q)q^{-1}(q-1)\E_{\rho''}&\text{if $\chi_q=1$,}\\
\epsilon\chi_{\stufe_1}(q) q^{-2}(q-1)\E_{\rho''}&\text{if $\chi_q^2=1$, $\chi_q\not=1$,}\\
0&\text{if $\chi_q^2\not=1$.}
\end{cases}$$
When $q=2$ we have
$\sum_Y\overline\chi_q(\det N)=4$ so the contribution is $\chi_{\stufe_1}(q)q^{-1}(q-1)\E_{\rho''}$.

Now say $\rank_qM'=1$; we can assume $q$ divides row 2 of $M'$.  Choose $E\in SL_2(\Z)$
so that $M'E\equiv \left(\begin{array}{cc} m_1'&0\\0&0\end{array}\right)\ (q)$; by symmetry,
$N'\,^tE^{-1}\equiv\left(\begin{array}{cc} n_1'&n_2'\\0&n_4'\end{array}\right)\ (q).$  Clearly
$E^{-1}Y\,^tE^{-1}$ varies over all symmetric matrices as $Y$ does, so
\begin{align*}
\sum_Y\overline\chi_q(\det N)&=q^2\sum_{u\,(q)}\overline\chi_q((n_1'-m_1'u)n_4')\\
&={\begin{cases} q^2(q-1)&\text{if $\chi_q=1$,}\\0&\text{otherwise.}\end{cases}}
\end{align*}
Thus the contribution from these terms is
$$\begin{cases} q^{-1}(q-1)\E_{\rho'}&\text{if $\chi_q=1$,}\\0&\text{otherwise.}\end{cases}$$

Finally, suppose $\rank_qM'=0$.  Then $\sum_Y\overline\chi_q(\det N)=q^3\overline\chi_q(\det N')$,
so the contribution is $\chi_{\stufe_1}(q)\chi_{\stufe_2}(q^2)\E_{\rho}.$ 
\end{proof}

\begin{prop}\label{proposition 3.8}  For $q$ a prime dividing $\stufe_2$, we have
$$\E_{\rho}|T_1(q^2)=\chi_{\stufe_0}(q^2)(q+1)q^{2k-3}\E_{\rho}.$$
\end{prop}

\begin{proof} From Proposition 2.1 of \cite{HW}, we know $2\E_{\rho}(\tau)|T_1(q^2)$ is the sum of terms
\begin{align*}
&p^{k-3}\overline\chi_{\rho}(M,N)\cdot\\
&\quad\det(M\left(\begin{array}{cc} \frac{1}{q}\\&1\end{array}\right) G^{-1}\tau
+N\left(\begin{array}{cc} q\\&1\end{array}\right)\,^tG+M\left(\begin{array}{cc} \frac{1}{q}\\&1\end{array}\right) Y\,^tG)^{-k}
\end{align*}
where $(M\ N)$ varies over $SL_2(\Z)$-equivalence class representatives
of pairs of $\stufe$-type $\rho$,
$G\in\G_1$, $Y=\left(\begin{array}{cc} y_1&y_2\\y_2&0\end{array}\right)$ with $y_1$ varying modulo $q^2$,
$y_2$ modulo $q$.  Recall that we can choose $G\equiv I\ (\stufe/q)$ and 
$Y\equiv0\ (\stufe/q).$

We know $M\left(\begin{array}{cc} \frac{1}{q}\\&1\end{array}\right)$ is never integral.  Adjust the equivalence
class representative for $(M\ N)$ so that
\begin{align*}
&(M'G\ N'\,^tG^{-1})=\\
&\quad\left(\begin{array}{cc} 1\\&q\end{array}\right)
\left(M\left(\begin{array}{cc} \frac{1}{q}\\&1\end{array}\right)\ N\left(\begin{array}{cc} q\\&1\end{array}\right)+M\left(\begin{array}{cc} \frac{1}{q}\\&1\end{array}\right)
Y\right)
\end{align*}
is an integral pair.  Since $\rank_qM'=2$, we have $(M',N')=1$.  Also,
$$\chi_{(\stufe_0,\stufe_1,\stufe_2/q)}(M',N')
=\chi_{\stufe_0\stufe_1}(q^2)\chi_{(\stufe_1,\stufe_1,\stufe_2/q)}(M,N),$$
and we know $\chi_{\stufe_1}^2=1$.

Reversing, choose $(M'\ N')$ of $\stufe$-type $\rho$, and set
\begin{align*}
&(M\ N)=\\
&\quad
\left(\begin{array}{cc} 1\\&\frac{1}{q}\end{array}\right) E\left(M'G\left(\begin{array}{cc} q\\&1\end{array}\right)
\ (N'\,^tG^{-1}-M'GY)\left(\begin{array}{cc} \frac{1}{q}\\&1\end{array}\right)\right),
\end{align*}
where $E\in\,^t\G_1$.  Write $EM'G=\left(\begin{array}{cc} m_1&m_2\\m_3&m_4\end{array}\right)$,
$EN'\,^tG^{-1}=\left(\begin{array}{cc} n_1&n_2\\n_3&n_4\end{array}\right)$.

For each choice of $E$, we need to choose the unique $G$ so that $q|m_4$
to ensure $M$ is integral.  Thus $q\nmid m_2m_3$.  Choose $y_2$ so that $n_4\equiv m_3y_2\ (q)$;
then choose $y_q$ so that $n_3\equiv m_3y_1+m_4y_2\ (q^2)$.  By symmetry, $m_1n_3+m_2n_4\equiv m_3n_1\ (q).$
Thus $$m_3n_1\equiv m_3m_3y_1+m_2m_3y_2\ (q),$$
so $n_1\equiv m_1y_1+m_2y_2\ (q)$, and hence $N$ is integral.  Also, $\det M=\det M'$,
so $\chi_{\rho}(M',N')=\chi_{\stufe_0}(q^2)\chi_{\rho}(M,N).$ 
\end{proof}

\begin{prop}\label{proposition 3.9}  For $q$ a prime dividing $\stufe_1$,
set $\rho'=(\stufe_0,\stufe_\frac{1}{q},q\stufe_2)$.  Then
$$\E_{\rho}|T_1(q^2)
={\begin{cases}
(\chi_{\stufe_0}(q^2)q^{2k-2}+\chi_{\stufe_2}(q^2)q)\E_{\rho}\\
\quad +q^{-1}(\chi_{\stufe/q}(q)q^{k-2}+\chi_{\stufe_2}(q^2))(q^2-1)\E_{\rho'}
&\text{if $\chi_q=1$,}\\
(\chi_{\stufe_0}(q^2)q^{2k-2}+\chi_{\stufe_2}(q^2)q)\E_{\rho}&\text{if $\chi_q\not=1$.}
\end{cases}}$$
\end{prop}

\begin{proof} Recall that we must have $\chi_q^2=1$ since $q|\stufe_1$.
Take $(M\ N)$ of $\stufe$-type $\rho$; assume $q$ divides row 2 of $M'$.

First suppose $M\left(\begin{array}{cc} \frac{1}{q}\\&1\end{array}\right)$ is not integral;
thus by symmetry, $$N=\left(\begin{array}{cc} *&*\\*&n_4\end{array}\right)$$ where $q\nmid n_4$.
Thus $(M',N')=1$ where
\begin{align*}
&(M'G\ N'\,^tG^{-1})=\\
&\left(\begin{array}{cc} q\\&1\end{array}\right)
\left(M\left(\begin{array}{cc} \frac{1}{q}\\&1\end{array}\right)\ 
N\left(\begin{array}{cc} q\\&1\end{array}\right)+M\left(\begin{array}{cc} \frac{1}{q}\\&1\end{array}\right) Y\right).
\end{align*}
We know $q\nmid M'$ and $q|\det M'$, so $\rank_qM'=1$. Also, $\chi_{(\stufe_0,\stufe_\frac{1}{q},\stufe_2)}(M',N')
=\chi_{\stufe_0}(q^2)\chi_{(\stufe_0,\stufe_\frac{1}{q},\stufe_2)}(M,N).$

Reversing, take $(M'\ N')$ of $\stufe$-type $\rho$; assume $q$ divides row 2 of $M'$.  Set
$$(M\ N)=\left(\begin{array}{cc} \frac{1}{q}\\&1\end{array}\right) E
\left(M'G\left(\begin{array}{cc} q\\&1\end{array}\right)\ (N'\,^tG^{-1}-M'GY)\left(\begin{array}{cc} \frac{1}{q}\\&1\end{array}\right)\right)$$
where $E\in\G_1$.  Write $EM'G=\left(\begin{array}{cc} m_1&m_2\\m_3&m_4\end{array}\right)$,
$EN'\,^tG^{-1}=\left(\begin{array}{cc} n_1&n_2\\n_3&n_4\end{array}\right)$.  For 1 choice of $E$, $q$ divides row 1 of $EM'$;
then to have $N$ integral we need $q|(n_1\ n_2)$, which is impossible since $(M',N')=1$.
So choose $E$ so that $q$ does not divide row 1 of $EM'$ ($q$ choices for $E$).  Thus
$q$ divides row 2 of $EM'$; to have $M$ integral, choose the unique $G$ so that
$q|m_2$; so $q\nmid m_1$, and by symmetry, $q|n_3$ (so $q\nmid n_4$).  Choose the unique
$y_2\ (q)$ so that $n_2\equiv m_1y_2\ (q)$ and choose the unique $y_1\ (q^2)$ so that
$n_1\equiv m_1y_1+m_2y_2\ (q^2)$.  Thus $(M,N)=1$, and
$\chi_{\rho}(M',N')=\chi_{\rho}(M,N)$.  So the contribution from these terms is
$\chi_{\stufe_0}(q^2)q^{2k-2}\E_{\rho}$.

Now suppose $M\left(\begin{array}{cc} \frac{1}{q}\\&1\end{array}\right)$ is integral
with $$\rank_q\left(M\left(\begin{array}{cc} \frac{1}{q}\\&1\end{array}\right)\ N\left(\begin{array}{cc} q\\&1\end{array}\right)\right)=2.$$
Thus
$M\equiv\left(\begin{array}{cc} m_1&m_2\\m_3&m_4\end{array}\right)\ (q)$ where 
$q|m_1,m_3,m_4$, $q\nmid m_2$.
By symmetry, $q$ divides row 2 of $N\left(\begin{array}{cc} q\\&1\end{array}\right).$  So
$\rank_q\left(M\left(\begin{array}{cc} \frac{1}{q}\\&1\end{array}\right)\ N\left(\begin{array}{cc} q\\&1\end{array}\right)\right)\ge 1$.
In the case this rank is 2, we have $q^2\nmid m_3$, and $\rank_qM'=2$ where
$$(M'G\ N'\,^tG^{-1})=\left(M\left(\begin{array}{cc} \frac{1}{q}\\&1\end{array}\right)\ N\left(\begin{array}{cc} q\\&1\end{array}\right)
+M\left(\begin{array}{cc} \frac{1}{q}\\&1\end{array}\right) Y\right)$$
Also, $\chi_{(\stufe_0,\stufe_\frac{1}{q},\stufe_2)}(M',N')
=\chi_{\stufe/q}(q)\chi_{(\stufe_0,\stufe_\frac{1}{q},\stufe_2)}(M,N)$.

Reversing, first choose $(M'\ N')$ of $\stufe$-type $(\stufe_0,\stufe_\frac{1}{q},q\stufe_2)$.  Set
$$(M\ N)=\left(M'G\left(\begin{array}{cc} q\\&1\end{array}\right)\ (N'\,^tG^{-1}-M'GY)\left(\begin{array}{cc} \frac{1}{q}\\&1\end{array}\right)\right).$$
Write $M'G=\left(\begin{array}{cc} m_1&m_2\\m_3&m_4\end{array}\right)$, $N'\,^tG^{-1}=\left(\begin{array}{cc} n_1&n_2\\n_3&n_4\end{array}\right)$.
Suppose $\rank_qM'=2$.  Then for each $G$, adjust the equivalence class representative
so that $q|m_4$ (so $q\nmid m_2m_3$).  Take $u,y_2$ so that
$$\overline M'\left(\begin{array}{cc} n_1\\n_3\end{array}\right)\equiv\left(\begin{array}{cc} u\\y_2\end{array}\right)\ (q);$$
set $y_1=u+qu'$ where $u'$ varies modulo $q$.  Then summing over corresponding $Y$, with
$n_3'=(n_3-m_3u-m_4y_2)/q$, we have
\begin{align*}
\sum_Y\chi_{(1,q,1)}(M,N)
&=\sum_{u'}\chi_q(-m_2)\chi_q(n_3'-m_3u')\\
&={\begin{cases} (q-1)&\text{if $\chi_q=1$,}\\0&\text{otherwise.}\end{cases}}
\end{align*}
Thus the contribution from these terms is
$\chi_{\stufe/q}(q)q^{k-3}(q^2-1)\E_{\rho'}$ if $\chi_q=1$, and 0 otherwise.

Suppose $M\left(\begin{array}{cc} \frac{1}{q}\\&1\end{array}\right)$ is integral, 
$\rank_q\left(M\left(\begin{array}{cc} \frac{1}{q}\\&1\end{array}\right)\ N\left(\begin{array}{cc} q\\&1\end{array}\right)\right)=1.$
Since $q\nmid m_2n_3$, $(M'\ N')$ is an integral coprime pair where
\begin{align*}
&(M'G\ N'\,^tG^{-1})=\\
&\quad\left(\begin{array}{cc} 1\\&\frac{1}{q}\end{array}\right) 
\left(M\left(\begin{array}{cc} \frac{1}{q}\\&1\end{array}\right)\ N\left(\begin{array}{cc} q\\&1\end{array}\right)+M\left(\begin{array}{cc} \frac{1}{q}\\&1\end{array}\right) Y\right).
\end{align*}
Note that $\rank_qM'\ge 1$.  So 
$$\chi_{(\stufe_0,\stufe_\frac{1}{q},\stufe_2)}(M',N')
=\chi_{\stufe_2}(q^2)\chi_{(\stufe_0,\stufe_\frac{1}{q},\stufe_2)}(M,N).$$

Reversing, take $(M'\ N')$ of $\stufe/q$-type $(\stufe_0,\stufe_\frac{1}{q},\stufe_2)$, $\rank_qM'\ge 1$.
Set
$$(M N)=\left(\begin{array}{cc} 1\\&q\end{array}\right) E\left(
M'G\left(\begin{array}{cc} q\\&1\end{array}\right)\ (N'\,^tG^{-1}-M'GY)\left(\begin{array}{cc} \frac{1}{q}\\&1\end{array}\right)\right),$$
$E\in\G_1$.  Write $EM'G=\left(\begin{array}{cc} m_1&m_2\\m_3&m_4\end{array}\right)$,
$EN'\,^tG^{-1}=\left(\begin{array}{cc} n_1&n_2\\n_3&n_4\end{array}\right)$.

To have $\rank_qM=1$, we need to choose $E,G$ so that $q\nmid m_2$, and to have $N$
integral with $(M,N)=1$, we need to choose $Y$ so that
$n_1\equiv m_1y_1+m_2y_2\ (q)$, $n_3\not\equiv m_3y_1+m_4y_2\ (q).$

First suppose $\rank_qM'=2$.  To have $\rank_qM=1$, for each $E$ we need to choose
$G$ so that $q\nmid m_2$ (for each $E$, this gives us $q$ choices for $G$).
If $q\nmid m_3$ then we choose $y_1$ freely; if $q|m_3$ then $q\nmid m_1m_4$, so we can
choose $y_2$ freely (subject to $n_3\not\equiv m_4y_2\ (q)$).  In either case, we get
\begin{align*}
\sum_Y\chi_{(1,q,1)}(M,N)&=\chi_q(-m_2)\sum_Y\chi_q(n_3-m_3y_1-m_4y_2)\\
&={\begin{cases} q(q-1)&\text{if $\chi_q=1$,}\\0&\text{otherwise.}\end{cases}}
\end{align*}
So the contribution from these terms is
$\chi_{\stufe_2}(q^2)q^{-1}(q^2-1)\E_{\rho'}$ if $\chi_q=1$, 0 otherwise.

Finally, suppose $\rank_qM'=1$; assume $q$ divides row 2 of $M'$.  We have
$q$ choices for $E$ so that $q$ does not divide row 1 of $EM'$
(and then $q$ divides row 2 of $EM'$); then we have $q$ choices for $G$ so that
$q\nmid m_2$.  By symmetry, $q\nmid n_3$.  Choose $y_1$ freely, then choose $y_2$ so
that $n_1\equiv m_1y_1+m_2y_2\ (q)$.  
So the contribution from these terms is $\chi_{\stufe_2}(q^2)q\E_1$. 
\end{proof}

\begin{prop}\label{proposition 3.10}  Let $q$ be a prime dividing $\stufe_0$.
Set 
$$\rho'=(\stufe_0/q,q\stufe_1,\stufe_2),\ \rho''=(\stufe_0/q,\stufe_1,q\stufe_2).$$
Then
\begin{align*}
&\E_{\rho}|T_1(q^2)\\
&\quad={\begin{cases} 
\chi_{\stufe_2}(q^2)(q+1)\E_{\rho}\\
\quad + q^{-1}(\chi_{\stufe/q}(q)q^{k-1}+\chi_{\stufe_2}(q^2))(q-1)\E_{\rho'}\\
\quad +\chi_{\stufe_2}(q^2)q^{-2}(q^2-1)\E_{\rho''}
&\text{if $\chi_q=1$,}\\
\chi_{\stufe_2}(q^2)(q+1)\E_{\rho}
+\epsilon \chi_{\stufe_2}(q^2)q^{-2}(q^2-1)\E_{\rho''}
&\text{if $\chi_q^2=1$, $\chi_q\not=1$,}\\
\chi_{\stufe_2}(q^2)(q+1)\E_{\rho}
&\text{if $\chi_q^2\not=1$.}
\end{cases}}
\end{align*}
\end{prop}

\begin{proof} 
Take $(M\ N)$ of $\stufe$-type $\rho$.  So $\rank_qM=0$,
$\rank_qM\left(\begin{array}{cc} \frac{1}{q}\\&1\end{array}\right)\le 1$.

First suppose $(M',N')=1$ where
$$(M'G\ N'\,^tG^{-1})
=\left(M\left(\begin{array}{cc} \frac{1}{q}\\&1\end{array}\right)\ 
N\left(\begin{array}{cc} q\\&1\end{array}\right)+M\left(\begin{array}{cc} \frac{1}{q}\\&1\end{array}\right) Y\right).$$
Since $\rank_qN\left(\begin{array}{cc} q\\&1\end{array}\right)=1$, we must have $\rank_qM'=1$.
Also, 
$$\chi_{(\stufe_0/q,\stufe_1,\stufe_2)}(M',N')
=\chi_{\stufe/q}(q)\chi_{(\stufe_0/q,\stufe_1,\stufe_2)}(M,N).$$

Reversing, take $(M'\ N')$ of $\stufe$-type $\rho'$; set
$$(M\ N)=\left(M'G\left(\begin{array}{cc} q\\&1\end{array}\right)
\ (N'\,^tG^{-1}-M'GY)\left(\begin{array}{cc} \frac{1}{q}\\&1\end{array}\right)\right).$$
Write $M'G=\left(\begin{array}{cc} m_1&m_2\\m_3&m_4\end{array}\right)$,
$N'\,^tG^{-1}=\left(\begin{array}{cc} n_1&n_2\\n_3&n_4\end{array}\right)$; we can assume that
$q|(m_3\ m_4)$.  To have $q|M$, choose the unique $G$ so that $q|m_2$; then
$q\nmid m_1$, and by symmetry $q|n_3$ (so $q\nmid n_4$).  To have $N$ integral,
choose $u$ so that $n_1\equiv m_1u\ (q)$ and set $y_1=u+qu'$.  
For each choice of $u',y_2$, we have $\det N\equiv x-m_1n_4u'\ (q)$ where
$x$ does not depend on $u'$.  Hence, fixing $y_2$,
$$\sum_{u'\,(q)}\overline\chi_q(x-m_1n_4u')=
\begin{cases} (q-1)&\text{if $\chi_q=1$,}\\ 0&\text{otherwise.}\end{cases}$$
So, letting $y_2$ vary modulo $q$, we see the contribution from these terms is
$$\begin{cases} \chi_{\stufe/q}(q)q^{k-3}\cdot q(q-1)\E_{\rho'}&\text{if $\chi_q=1$,}\\
0&\text{otherwise.}\end{cases}$$

Now say $\rank_q\left(M\left(\begin{array}{cc} \frac{1}{q}\\&1\end{array}\right)\ N\left(\begin{array}{cc} q\\&1\end{array}\right)\right)=1;$
adjust the equivalence class representative $(M\ N)$ as necessary so that
\begin{align*}
&(M'G\ N'\,^tG^{-1})=\\
&\quad\left(\begin{array}{cc} 1\\&\frac{1}{q}\end{array}\right)\left(M\left(\begin{array}{cc} \frac{1}{q}\\&1\end{array}\right)
\ N\left(\begin{array}{cc} q\\&1\end{array}\right)+M\left(\begin{array}{cc} \frac{1}{q}\\&1\end{array}\right) Y\right)
\end{align*}
is integral.  Then $(M',N')=1$ since $\rank_qN'=2$.  Also,
$$\chi_{(\stufe_0/q,\stufe_1,\stufe_2)}(M',N')
=\chi_{\stufe_2}(q^2)\chi_{(\stufe_0/q,\stufe_1,\stufe_2)}(M,N).$$

Reversing, take $(M'\ N')$ of $\stufe/q$-type $(\stufe_0/q,\stufe_1,\stufe_2)$, and set
$$(M\ N)=\left(\begin{array}{cc} 1\\&q\end{array}\right) E
\left(M'G\left(\begin{array}{cc} q\\&1\end{array}\right)\ (N'\,^tG^{-1}-M'GY)\left(\begin{array}{cc} \frac{1}{q}\\&1\end{array}\right)\right),$$
$E\in\G_1$.  Write $EM'G=\left(\begin{array}{cc} m_1&m_2\\m_3&m_4\end{array}\right)$,
$EN'\,^tG^{-1}=\left(\begin{array}{cc} n_1&n_2\\n_3&n_4\end{array}\right).$

Say $\rank_qM'=2$.  To have $q|M$, for each $E$ we need to choose the unique $G$ so that
$q|m_2$; thus $q\nmid m_1m_4$.  To have $N$ integral, choose $u$ so that $n_1\equiv m_1u\ (q);$
set $y_1=u+qu'$, $u'$ varying modulo $q$.  By symmetry, $m_1n_3\equiv m_3n_1+m_4n_2\ (q),$
so $n_3\equiv\overline m_1m_3n_1+\overline m_1m_4n_2\ (q);$ hence
$\det N\equiv -m_1m_4(\overline m_1n_2-y_2)^2\ (q)$.  Thus summing over $Y$ where
$y_1=u+qu'$,
\begin{align*} \sum_Y\overline\chi_q(\det N)&=q\overline\chi_q(-m_1m_4)
\sum_{y_2\,(q)}\chi_q^2(\overline m_1n_2-y_2)\\
&={\begin{cases} q(q-1)\chi_q(-m_1m_4)&\text{if $\chi_q^2=1$,}\\0&\text{otherwise.}\end{cases}}
\end{align*}
So the contribution from these terms is
$$\begin{cases} 
\chi_{\stufe_2}(q^2)q^{-2}(q^2-1)\E_{\rho''}&\text{if $\chi_q=1$,}\\
\chi_{\stufe_2}(q^2)q^{k-3}\cdot q^{-k}\cdot q(q^2-1)\epsilon \E_{\rho''}
&\text{if $\chi_q\not=1$, $\chi_q^2=1$,}\\
0&\text{otherwise.}\end{cases}$$

Now say $\rank_qM'=1$; we can assume $q$ divides row 2 of $M'$.  For $q$ choices of $E$
we have $q$ dividing row 2 of $EM'$.  Then to have $q|M$, choose the unique $G$ so that
$q|m_2$; so $q\nmid m_1$ and by symmetry, $q|n_3$.  But then $q$ divides row 2 of $(M\ N)$,
so $(M,N)\not=1$.  With the 1 other choice of $E$, we have $q$ dividing row 1 of $EM'$.
Then to have $N$ integral, choose the unique $G$ so that $q|n_1$.  Then $q\nmid n_2$ (since
$(M',N')=1$), so by symmetry, $q|m_4$ (and hence $q\nmid m_3$).  Thus
\begin{align*}\sum_Y\overline\chi_q(\det N)&=\sum_Y\overline\chi_q(-n_2(n_3-m_3y_1))\\
&={\begin{cases} q^3\text{if $\chi_q=1$,}\\0&\text{otherwise.}\end{cases}}
\end{align*}
So the contribution from these terms is
$$\begin{cases} \chi_{\stufe_2}(q^2)q^{k-3}\cdot q^{-k}\cdot q^2(q-1)\E_{\rho'}
&\text{if $\chi_q=1$,}\\0&\text{otherwise.}\end{cases}$$

Finally, say $\rank_qM'=0$.  So to have $N$ integral,
for each $E$ we need to choose the unique $G$ so that
$q|n_1$.  Then $q\nmid n_2n_3$, and for each choice of $Y$ we have $N$ integral
with $\det N\equiv\det N'\ (q).$  Hence the contribution from these terms is
$\chi_{\stufe_2}(q^2)q^{k-3}\cdot q^{-k}\cdot q^3(q+1)\E_{\rho}.$

\end{proof}

With these results, we now construct a basis of simultaneous Hecke eigenforms.

\begin{thm}\label{Theorem 3.11}  Take square-free $\stufe\in\Z_+$ 
and a Dirichlet character $\chi$ modulo $\stufe$ so that $\chi(-1)=(-1)^k$.
There is a basis
$$\{\widetilde\E_{(\stufe_0,\stufe_1,\stufe_2)}:\ \stufe_0\stufe_1\stufe_2=\stufe,\ \chi_{\stufe_1}^2=1\ \}$$
for the space $\Eis_k^{(2)}(\stufe,\chi)$ of degree 2 Siegel Eisenstein series
of weight $k$, level $\stufe$, character $\chi$ so that for any prime $p$,
$\widetilde\E_{(\stufe_0,\stufe_1,\stufe_2)}|T(p)
=\lambda(p)\widetilde\E_{(\stufe_0,\stufe_1,\stufe_2)}$
and 
$\widetilde\E_{(\stufe_0,\stufe_1,\stufe_2)}|T_1(p^2)
=\lambda_1(p^2)\widetilde\E_{(\stufe_0,\stufe_1,\stufe_2)}$ where
$$\lambda(p)=
\begin{cases} 
%\chi_{\stufe_0}(p^2)\chi_{\stufe_1}(p)p^{2k-3}+\chi_{\stufe_0\stufe_2}(p)p^{k-2}(p+1)
%+\chi_{\stufe_1}(p)\chi_{\stufe_2}(p^2)
(\chi_{\stufe_0\stufe_1}(p)p^{k-1}+\chi_{\stufe_2}(p))(\chi_{\stufe_0}(p)p^{k-2}+\chi_{\stufe_1\stufe_2}(p)) 
&\text{if $p\nmid\stufe$,}\\
\chi_{\stufe_0}(p^2)\chi_{\stufe_1}(p)p^{2k-3}
&\text{if $p|\stufe_2$,}\\
\chi_{\stufe_0\stufe_2}(p)p^{k-1}
&\text{if $p|\stufe_1$,}\\
\chi_{\stufe_1}(p)\chi_{\stufe_2}(p^2)
&\text{if $p|\stufe_0$,}\end{cases}$$
and
$$\lambda_1(p^2)
=\begin{cases} (p+1)(\chi_{\stufe_0}(p^2)p^{2k-3}+\chi(p)p^{k-3}(p-1)+\chi_{\stufe_2}(p^2))
&\text{if $p\nmid\stufe$,}\\
\chi_{\stufe_0}(p^2)(p+1)p^{2k-3}
&\text{if $p|\stufe_2$,}\\
\chi_{\stufe_0}(p^2)p^{2k-3}+\chi_{\stufe_2}(p^2)p
&\text{if $p|\stufe_1$,}\\
\chi_{\stufe_2}(p^2)(p+1)
&\text{if $p|\stufe_0$.}\end{cases}$$
\end{thm}

\begin{proof} Fix a partition $\rho=(\stufe_0,\stufe_1,\stufe_2)$ of $\stufe$.
For $q$ an odd prime dividing $\stufe$, set $\epsilon=\left(\frac{-1}{q}\right)$, and set
\begin{align*}
a_{\rho}(q)&={\begin{cases} -\frac{\chi_{\stufe_1\stufe_2}(q)q^{-1}(q-1)}{\chi_{\stufe_0}(q)q^{k-1}-\chi_{\stufe_1\stufe_2}(q)}
&\text{if $\chi_q=1$,}\\0&\text{otherwise;}\end{cases}}\\
b_{\rho}(q)&={\begin{cases} -\frac{ \chi_{\stufe_2}(q^2)
q^{-1}(q-1)(\chi_{\stufe_0}(q)q^{k-3}-\chi_{\stufe_1\stufe_2}(q))}
{(\chi_{\stufe_0}(q)q^{k-1}-\chi_{\stufe_1\stufe_2}(q))(\chi_{\stufe_0}(q^2)q^{2k-3}-\chi_{\stufe_2}(q^2))}
&\text{if $\chi_q=1$,}\\
-\frac{\epsilon \chi_{\stufe_2}(q^2) q^{-2}(q-1)}
{\chi_{\stufe_0}(q^2)q^{2k-3}-\chi_{\stufe_2}(q^2)}&\text{if $\chi_q\not=1$, $\chi_q^2=1$,}\\0&\text{otherwise;}
\end{cases}}\\
c_{\rho}(q)&={\begin{cases} -\frac{\chi_{\stufe_2}(q)q^{-2}(q^2-1)}{\chi_{\stufe_0\stufe_1}(q)q^{k-2}-\chi_{\stufe_2}(q)}
&\text{if $\chi_q=1$,}\\0&\text{otherwise.}\end{cases}}
\end{align*}
Extend these functions multiplicatively, and set
$$\widetilde\E_{\rho}=\sum_{{Q_0Q_0'|\stufe_0}\atop{Q_1|\stufe_1}}
a_{\rho}(Q_0)b_{\rho}(Q_0')c_{\rho}(Q_1)\E_{(\stufe_0/(Q_0Q_0'),Q_0\stufe_1/q,Q_0'Q_1\stufe_2)}.$$
Since $a_{\rho}(Q_0)=0$ unless $\chi_{Q_0}=1$, $b_{\rho}(Q_0')=0$ unless $\chi_{Q_0'}^2=1$,
and $c_{\rho}(Q_1)=0$ unless $\chi_{Q_1}=1$, Propositions 3.3, 3.4, 3.5, and 3.8 show that
$\widetilde\E_{\rho}$ is an eigenform for all $T(p)$, $T_1(p^2)$ where $p\nmid \stufe_0\stufe_1$,
with eigenvalues as claimed in the theorem.
For a prime $q|\stufe_1$, write
\begin{align*}
\widetilde\E_{\rho}
&\quad=
\sum_{{Q_0Q_0'|\stufe_0}\atop{Q_1|\stufe_1/q}}a_{\rho}(Q_0)b_{\rho}(Q_0')c_{\rho}(Q_1)\\
&\quad\cdot \left(\E_{(\stufe_0/(Q_0Q_0'),Q_0\stufe_1/q,Q_0'Q_1\stufe_2)}
+c_{\rho}(q) \E_{(\stufe_0/(Q_0Q_0'),Q_0\stufe_1/(qQ_1),qQ_0'Q_1\stufe_2)}\right).
\end{align*}
Propositions 3.5, 3.6, 3.8 and 3.9 show that $\widetilde\E_{\rho}$ is an eigenform for $T(q)$
and $T_1(q^2)$, with eigenvalues
as claimed.  
For $q|\stufe_0$, using Propositions 3.5 through 3.10 and a similar rearrangement of the sum defining
$\widetilde\E_{\rho}$, we find $\widetilde\E_{\rho}$ is an eigenform for $T(q)$ and $T_1(q^2)$, with
eigenvalues as claimed. 
\end{proof}

Note that for $q$ a prime dividing $\stufe$ with $\chi_q^2=1$, 
Propositions 3.6, 3.7, 3.9 and 3.10  give us Hecke relations among Eisenstein series.  In particular,
when $\cond\chi^2<\stufe$ we can generate some of the Eisenstein series from $\E_{(\stufe,1,1)}$.  To see this,
let $q$ be a prime dividing $\stufe$ so that $\chi_q^2=1$.  If $\chi_q=1$, set
$c(q)=\frac{q^2}{(q-1)(\chi_{\stufe/q}(q)q^k-1)}$,
$$S_1(q)=c(q)\left[T_1(q^2)-q^{-1}(q+1)T(q)-q^{-1}(q^2-1)\right]$$
and 
$$S_2(q)=c(q)\left[(\chi_{\stufe/q}(q)q^{k-1}+1)T(q)-T_1(q^2)-q(\chi_{\stufe/q}(q)q^{k-2}-1)\right];$$
if $\chi_q\not=1$, set
$$S_2(q)=\frac{\left(\frac{-1}{q}\right)q^2}{(q-1)}\left[T(q)-1\right].$$
Extending these maps multiplicatively, we have the following.

\begin{thm}\label{Theorem 3.12}  Suppose $\stufe$ is square-free,
$\stufe_0\stufe_1\stufe_2=\stufe$, $\chi_{\stufe_1}=1$,
and $\chi_{\stufe_2}^2=1$.  Then 
$$\E_{(\stufe_0,\stufe_1,\stufe_2)}
=\E_{(\stufe,1,1)}|S_1(\stufe_1)S_2(\stufe_2).$$
In particular, when $\chi=1$,
$$\{\E_{(\stufe,1,1)}|S_1(\stufe_1)S_2(\stufe_2):\ \stufe_1\stufe_2|\stufe\ \}$$
is a basis for $\Eis_k^{(2)}(\stufe,1)$.
\end{thm}

\begin{proof}  When $\chi_q=1$, we use Propositions 3.7, 3.10 to solve for $\E_{(\stufe_0/q,q\stufe_1,\stufe_2)}$
and for $\E_{(\stufe_0/q,\stufe_1,q\stufe_2)}$ in terms of $\E_{(\stufe_0,\stufe_1,\stufe_2)}$, presuming $q|\stufe_0$.
When $\chi_q\not=1$ we use Proposition 3.7 to get $\E_{(\stufe_0/q,\stufe_1,q\stufe_2)}$ in terms of $\E_{(\stufe_0,\stufe_1,\stufe_2)}$, 
again presuming $q|\stufe_0$.  Now a simple induction argument yields the result.

\end{proof}

\smallskip\noindent{\bf Remarks.}
(1) When $f$ is a Siegel modular form 
with Fourier coeffients
$a(T)$, we have $a(\,^tGTG)=a(T)$ for all $G\in GL_2(\Z)$ if $k$ is even, and for all $G\in SL_2(\Z)$ if $k$ is odd.  Thus we can consider the Fourier series
of $f$ to be supported on lattices $\Lambda$ equipped with a positive, semi-definite quadratic
form given by $T$ (relative to some basis), with $\Lambda$ oriented if $k$ is odd; for such $\Lambda$ we set $a(\Lambda)=a(T)$.  Then by Theorem 6.1 of [4],
with $QP$ square-free, the $\Lambda$th coefficient of $f|T_1(Q^2)T(P)$ is
$$\sum_{{Q\Lambda\subset\Omega\subset\Lambda}\atop{[\Lambda:\Omega]=Q}}a\left(\Omega^{P}\right),$$
where $\Omega^{P}$ denotes the lattice $\Omega$ whose quadratic form has been scaled
by $P$.  (Note that an orientation on $\Lambda$ induces an orientation on $\Omega\subset\Lambda$.)

(2) With $\E$ the degree 2 Eisenstein series of level 1, we have
$$\E(\tau)=\sum_{\stufe_0\stufe_1\stufe_2=\stufe}\E_{(\stufe_0,\stufe_1,\stufe_2)}.$$
Thus formulas for the Fourier coefficients of $\E$ together with our Hecke relations can be used to generate Fourier coefficients
of all degree 2, square-free level $\stufe$ Eisenstein series with trivial character.

\end{document}